\documentclass[12pt]{amsart}
\usepackage{amscd,amssymb,hyperref}
\usepackage{graphicx}
\usepackage{graphics}
\usepackage{epsfig}
\usepackage{picinpar}
\usepackage{wrapfig}
\usepackage{amsmath}
\usepackage{color}
\usepackage{pstricks,pst-node,pst-text,pst-3d}
\usepackage{pst-plot}
\usepackage{verbatim}

\newtheorem{thm}{Theorem}[section]
\newtheorem{lem}[thm]{Lemma}

\newtheorem{prop}[thm]{Proposition}
\newtheorem{question}[thm]{Question}

\def\square{\vbox{
      \hrule height 0.4pt
      \hbox{\vrule width 0.4pt height 5.5pt \kern 5.5pt \vrule width 0.4pt}
      \hrule height 0.4pt}}

\def\ch\mathrm{c h}

\newcommand{\VAP}{\mathrm{VAP}}

\numberwithin{equation}{section}

\begin{document}

\title[Lifting theorem for the virtual  pure braid groups]{Lifting theorem for the virtual pure braid groups}

\author[V. Bardakov]{Valeriy G. Bardakov}
\address{Sobolev Institute of Mathematics, Novosibirsk 630090, Russia,}
\address{Novosibirsk State University, Novosibirsk 630090, Russia,}
\address{Novosibirsk State Agrarian University, Dobrolyubova street, 160,  Novosibirsk 630039, Russia,}
\email{bardakov@math.nsc.ru}

%\author{Roman Mikhailov}
%\address{Laboratory of Modern Algebra and Applications, St. Petersburg State University, 14th Line, 29b,
%Saint Petersburg, 199178 Russia and St. Petersburg Department of
%Steklov Mathematical Institute} \email{rmikhailov@mail.ru}

\author{Jie Wu }
\address{School of Mathematical Sciences, Hebei Normal University, Shijiazhuang, Hebei 050024, China} \email{matwuj@nus.edu.sg}
\urladdr{www.math.nus.edu.sg/\~{}matwujie}
\thanks{The main result is supported by the Russian Science Foundation grant no. 16-11-10073.}

\begin{abstract} In this article we prove theorem on Lifting for the set of virtual pure braid groups. This theorem says that if we know presentation of virtual pure braid group $VP_4$, then we can find presentation of $VP_n$ for  arbitrary $n > 4$. Using this theorem we find the set of generators and defining relations for simplicial group $T_*$ which was defined in \cite{BW}. We find a decomposition of the Artin pure braid group $P_n$ in semi-direct product of free groups in the cabled generators.
\end{abstract}

\maketitle

\section{Introduction}

The operation cabling for classical braids studied in~\cite{CW}.
 For virtual pure braid group $VP_n$ this operation  gives new generators for $VP_n$ (see \cite{BW}). It was proved that for $n\geq 3$, the group $VP_n$ is generated by the $n$-strand virtual braids obtained by taking $(k,l)$-cabling on the standard generators $\lambda_{1,2}$ and $\lambda_{2,1}$ of $VP_2$ together with adding trivial strands $n-k-l$ to the end for $1\leq k\leq n-1$ and $2\leq k+l\leq n$, where a $(k,l)$-cabling on a $2$-strand virtual braid means to take $k$-cabling on the first strand and $l$-cabling on the second strand.

Different from the classical situation~\cite{CW} that the $n$-strand braids cabled from the standard generator $A_{1,2}$ for $P_2$ generates a free group of rank $n-1$, the subgroup of $VP_n$ generated by $n$-strand virtual braids cabled from $\lambda_{1,2}$ and $\lambda_{2,1}$, which is denoted by $T_{n-1}$, is no longer free for $n\geq3$.
%The group $T_{n-1}$ has $2(n-1)$ generators $a_{k,n-k}$ and $b_{k,n-k}$, $1\leq k\leq n-1$, where $a_{k,n-k}$ (or $b_{k,n-k}$) is obtained by taking $k$-cabling on the first strand and $(n-k)$-cabling on the second strand of $\lambda_{1,2}$ (or $\lambda_{2,1}$).
For the first nontrivial case that $n=3$, a presentation of $T_2$ has been explored with producing a decomposition theorem for $VP_3$ using cabled generators~\cite{BMVW}.

In the present  article  we continue to study $VP_n$ in cabled generators, which we started in \cite{BW}. We find some sufficient condition under which a simplicial group $G_*$ is contractible. In particular, we prove that the simplicial group $VAP_* = \{VP_i\}_{i=1,2,\ldots}$ is contractible.
Also, we prove the lifting theorem for the virtual pure braid groups. From this theorem follows that if we know the structure of $VP_4$, $T_3$ or $P_4$, then using degeneracy maps we can find the structure of $VP_n$, $T_n$ or $P_n$ for all bigger $n$. On the other side we prove that if we know a presentation of $VP_n$, $n \geq 4$, then conjugated it by elements $\rho_n$, $\rho_n \rho_{n-1}$, $\ldots$, $\rho_n \rho_{n-1} \ldots \rho_1 \in VB_{n+1}$ we can find the presentation of $VP_{n+1}$.

The article is organized as follows. In Section \ref{virt}, we give a review on braid groups and virtual braid groups. The simplicial structure on virtual pure braid groups will be discussed in Section~\ref{simp}. Is Section \ref{lift} we prove the lifting theorem.
In Section ~\ref{s41}, we discuss the cabling operation on classical pure braid group $P_n$ as subgroup of$VP_n$. We know two types of decompositions of $P_n$ as semi-direct products (see, for example, \cite{B1}). In Section \ref{s41} we construct new decomposition of this type in terms of the cabled generators.
In the last Section \ref{fin} we formulate some questions for further research.

\subsection{Acknowledgements}
This article was written when the first author visited  College of Mathematics and information Science Hebei Normal University. He thanks the administration for good working conditions.

\section{Braid and virtual braid groups} \label{virt}

\subsection{Braid group} The braid group $B_n$ on $n$ strings  is generated by $\sigma_1,\,  \sigma_2, \, \ldots , \, \sigma_{n-1}$ and is defined by relations
\begin{align*}
& \sigma_i \sigma_{i+1} \sigma_i = \sigma_{i+1} \sigma_i \sigma_{i+1},~~~i = 1, 2, \ldots, n-2, \\
& \sigma_i \sigma_{j} = \sigma_j \sigma_{i},~~|i-j|>1.
%\label{relation},
\end{align*}

Let $S_n$, $n \geq 1$ be the symmetric group which is generated by $\rho_1, \, \rho_2, \, \ldots , \, \rho_{n-1}$ and is defined by relations
\begin{align*}
& \rho_i^2 = 1,~~~i = 1, 2, \ldots, n-1,\\
& \rho_i \rho_{i+1} \rho_i = \rho_{i+1} \rho_i \rho_{i+1},~~~i = 1, 2, \ldots, n-2,\\
& \rho_i \rho_{j} = \rho_j \rho_{i},~~|i-j|>1.
%\label{relation},
\end{align*}

There is a homomorphism $B_n \to S_n$, which sends $\sigma_i$ to $\rho_i$. Its kernel is the pure braid group $P_n$. This group is generated by elements $A_{i,j}$, $1 \leq i < j \leq n$, where
$$
A_{i,i+1} = \sigma_i^2,
$$
$$
A_{i,j} = \sigma_{j-1} \sigma_{j-2} \ldots \sigma_{i+1} \sigma_i^2 \sigma_{i+1}^{-1} \ldots \sigma_{j-2}^{-1} \sigma_{j-1}^{-1},~~~i+1 < j \leq n,
$$
and is defined by relations (where $\varepsilon = \pm 1$):
\begin{align*}
& A_{ik}^{-\varepsilon} A_{kj} A_{ik}^{\varepsilon} = (A_{ij} A_{kj})^{\varepsilon} A_{kj} (A_{ij} A_{kj})^{-\varepsilon},\\
& A_{km}^{-\varepsilon} A_{kj} A_{km}^{\varepsilon} = (A_{kj} A_{mj})^{\varepsilon} A_{kj} (A_{kj} A_{mj})^{-\varepsilon},~~m < j, \\
& A_{im}^{-\varepsilon} A_{kj} A_{im}^{\varepsilon} = [A_{ij}^{-\varepsilon}, A_{mj}^{-\varepsilon}]^{\varepsilon} A_{kj} [A_{ij}^{-\varepsilon}, A_{mj}^{-\varepsilon}]^{-\varepsilon}, ~~i < k < m,\\
& A_{im}^{-\varepsilon} A_{kj} A_{im}^{\varepsilon} = A_{kj}, ~~k < i, m < j~\mbox{or}~ m < k,
\end{align*}
Here and further  $[a,b] = a^{-1} b^{-1} a b$ is the commutator of $a$ and $b$.

There is an epimorphism of $P_n$ to $P_{n-1}$ what is removing of the $n$-th string. Its kernel $U_n = \langle A_{1n}, A_{2n}, \ldots, A_{n-1,n} \rangle$ is a free group of rank $n-1$ and $P_n = U_n \leftthreetimes P_{n-1}$ is a semi-direct product of $U_n$ and $P_{n-1}$. Hence,
$$
P_n = U_n \leftthreetimes (U_{n-1} \leftthreetimes (\ldots \leftthreetimes (U_3 \leftthreetimes U_2)) \ldots ),
$$
is a semi-direct product of free groups and $U_2 = \langle A_{12}\rangle$ is the infinite cyclic  group.

\subsection{Virtual braid group} \label{virt1}

The virtual braid group $VB_n$ is generated by elements
$$
\sigma_1,\,  \sigma_2, \, \ldots , \, \sigma_{n-1}, \, \rho_1, \, \rho_2, \, \ldots , \, \rho_{n-1},
$$
where $\sigma_1,\,  \sigma_2, \, \ldots , \, \sigma_{n-1}$ generate the classical braid group $B_n$ and
the elements $\rho_1$,  $\rho_2$,  $\ldots $,  $\rho_{n-1}$ generate the symmetric group
$S_n$. Hence, $VB_n$ is defined by relations of $B_n$, relations of $S_n$
and mixed relation:
$$
\sigma_i \rho_j = \rho_j \sigma_i,~~~|i-j| > 1,
$$
$$
\rho_i \rho_{i+1} \sigma_i = \sigma_{i+1} \rho_i \rho_{i+1}~~~i = 1, 2, \ldots, n-2.
$$

As for the classical braid groups there exists the canonical
epimorphism of $VB_n$ onto the symmetric group $VB_n\to S_n$ with the
kernel called the {\it virtual pure  braid group} $VP_n$. So we have a
short exact sequence
\begin{equation*}
1 \to VP_n \to VB_n \to S_n \to 1.
\end{equation*}
Define the following elements in $VP_n$:
$$
\lambda_{i,i+1} = \rho_i \, \sigma_i^{-1},~~~
\lambda_{i+1,i} = \rho_i \, \lambda_{i,i+1} \, \rho_i = \sigma_i^{-1} \, \rho_i,
~~~i=1, 2, \ldots, n-1,
$$
$$
\lambda_{ij} = \rho_{j-1} \, \rho_{j-2} \ldots \rho_{i+1} \, \lambda_{i,i+1} \, \rho_{i+1}
\ldots \rho_{j-2} \, \rho_{j-1},
$$
$$
\lambda_{ji} = \rho_{j-1} \, \rho_{j-2} \ldots \rho_{i+1} \, \lambda_{i+1,i} \, \rho_{i+1}
\ldots \rho_{j-2} \, \rho_{j-1}, ~~~1 \leq i < j-1 \leq n-1.
$$
It is shown in \cite{B} that the group $VP_n$, $n\geq 2$ admits a
presentation with the  generators $\lambda_{ij},\ 1\leq i\neq j\leq n,$
and the following relations:
\begin{align}
& \lambda_{ij}\lambda_{kl}=\lambda_{kl}\lambda_{ij} \label{rel},\\
&
\lambda_{ki}\lambda_{kj}\lambda_{ij}=\lambda_{ij}\lambda_{kj}\lambda_{ki}
\label{relation},
\end{align}
where distinct letters stand for distinct indices.

Like the classical pure braid groups, groups $VP_n$ admit a
semi-direct product decompositions \cite{B}: for $n\geq 2,$ the
$n$-th virtual pure braid group can be decomposed as
\begin{equation}
VP_n=V_{n-1}^*\rtimes VP_{n-1},~~n \geq 2,
\label{eq:s_d_dec}
\end{equation}
where $V_{n-1}^*$ is a  subgroup of $VP_{n}$, $V_1^* = F_2$, $VP_1$ is supposed
to be the trivial group.

\section{Simplicial groups} \label{simp}

\subsection{Simplicial sets and simplicial groups} Recall the definition of simplicial groups (see \cite[p.~300]{MP} or \cite{BCWW}). A sequence of sets $X_* = \{ X_n \}_{n \geq 0}$  is called a
{\it simplicial set} if there are face maps:
$$
d_i : X_n \longrightarrow X_{n-1} ~\mbox{for}~0 \leq i \leq n
$$
and degeneracy  maps
$$
s_i : X_n \longrightarrow X_{n+1} ~\mbox{for}~0 \leq i \leq n,
$$
that are  satisfy the following simplicial identities:
\begin{enumerate}
\item $d_i d_j = d_{j-1} d_i$ if $i < j$,
\item $s_i s_j = s_{j+1} s_i$ if $i \leq j$,
\item $d_i s_j = s_{j-1} d_i$ if $i < j$,
\item $d_j s_j = id = d_{j+1} s_j$,
\item $d_i s_j = s_{j} d_{i-1}$ if $i > j+1$.
\end{enumerate}
Here $X_n$ can be geometrically viewed as the set of $n$-simplices including all possible degenerate simplices.

A {\it simplicial group} is a simplicial set $X_*$ such that each $X_n$ is a group and all face and degeneracy operations are group homomorphism.
 Let $G_*$ be a simplicial group. The \textit{Moore cycles} $\mathrm{Z}_n(G_*)\leq G_n$ is defined by
$$
\mathrm{Z}_n(G_*)=\bigcap_{i=0}^n\mathrm{Ker}(d_i\colon G_n\to G_{n-1})
$$
and the \textit{Moore boundaries} $\mathcal{B}_n(G_*)\leq G_n$ is defined by
$$
\mathcal{B}_n(G_*)=d_0\left(\bigcap_{i=1}^{n+1}\mathrm{Ker}(d_i\colon G_{n+1}\to G_n)\right).
$$
Simplicial identities guarantees that $\mathcal{B}_n(G_*)$ is a (normal) subgroup of $\mathrm{Z}_n(G_*)$. The \textit{Moore homotopy group} $\pi_n(G_*)$ is defined by
$$
\pi_n(G_*)=\mathrm{Z}_n(G_*)/\mathcal{B}_n(G_*).
$$
It is a classical result due to J. C. Moore ~\cite{Moore} that $\pi_n(G_*)$ is isomorphic to the $n$-th homotopy group of the geometric realization of $G_*$.

\subsection{Simplicial group  on virtual pure braid groups}
By using the same ideas in the work~\cite{BCWW,CW} on the classical braids, in \cite{BW} was introduced a simplcial group
$$
\VAP_* :\ \ \ \ldots\ \begin{matrix}\longrightarrow\\[-3.5mm] \ldots\\[-2.5mm]\longrightarrow\\[-3.5mm]
\longleftarrow\\[-3.5mm]\ldots\\[-2.5mm]\longleftarrow \end{matrix}\ VP_4 \ \begin{matrix}\longrightarrow\\[-3.5mm]\longrightarrow\\[-3.5mm]\longrightarrow\\[-3.5mm]\longrightarrow\\[-3.5mm]\longleftarrow\\[-3.5mm]
\longleftarrow\\[-3.5mm]\longleftarrow
\end{matrix}\ VP_3\ \begin{matrix}\longrightarrow\\[-3.5mm] \longrightarrow\\[-3.5mm]\longrightarrow\\[-3.5mm]
\longleftarrow\\[-3.5mm]\longleftarrow \end{matrix}\ VP_2\ \begin{matrix} \longrightarrow\\[-3.5mm]\longrightarrow\\[-3.5mm]
\longleftarrow \end{matrix}\ VP_1$$
on pure virtual braid groups with $\VAP_n=VP_{n+1}$, the face homomorphism
$$
d_i : \VAP_n=VP_{n+1} \longrightarrow \VAP_{n-1}=VP_n
$$
given by deleting $(i+1)$th strand for $0\leq i\leq n$, and the degeneracy homomorphism
$$
s_i : \VAP_n=VP_{n+1} \longrightarrow \VAP_{n+1}=VP_{n+2}
$$
given by doubling the $(i+1)$th strand for $0\leq i\leq n$.

Let $\iota_n\colon VP_n\to VP_{n+1}$ be the inclusion. Geometrically $\iota_n$ is the group homomorphism by adding a trivial strand on the end. From geometric information, we have the following formulae:
\begin{equation}\label{formula1}
s_j\iota_n=\iota_{n+1}s_j\colon VP_n\longrightarrow VP_{n+1} \textrm{ for } 0\leq j\leq n-1,
\end{equation}
\begin{equation}\label{formula2}
d_j\iota_n=\left\{
\begin{array}{lcl}
\iota_{n-1}d_j&\textrm{ if }& j<n,\\
\mathrm{id}&\textrm{ if }&j=n.\\
\end{array}\right.
\end{equation}
From the above formulae, the inclusion
$\iota_n\colon VP_n\to VP_{n+1}
$
gives an extra operation on the simplicial group $\mathrm{VAP}_*$ so that the simplicial identities still hold by regarding $\iota_n$ as \textit{extra degeneracy}
$$
s_n=\iota_n\colon \mathrm{VAP}_{n-1}=VP_n\longrightarrow \mathrm{VAP}_n=VP_{n+1}.
$$
Motivated from this example, a simplicial group $G_*$ is called \textit{conic} if there exists an extra degeneracy homomorphism $s_n\colon G_{n-1}\to G_n$ so that simplicial identities (including formulae involving $s_n$) hold.

\begin{prop}
Any conic simplicial group $G_*$ is contractible.
\end{prop}
\begin{proof}
Let $x\in \mathrm{Z}_n(G_*)$ be a Moore cycle, that is $x\in G_n$ with $d_jx=1$ for $0\leq j\leq n$. Note that we have the extra operation $s_{n+1}\colon G_n\to G_{n+1}$. Let $y=s_{n+1}x\in G_{n+1}$. Then
$$
d_jy=d_js_{n+1}x=s_{n}d_jx=s_n(1)=1
$$
for $0\leq j\leq n$ and
$$
d_{n+1}y=d_{n+1}s_{n+1}x=x.
$$
It follows that $x$ is a Moore boundary. Thus $\pi_n(G_*)=0$ for all $n$, and so $G_*$ is contractible.
\end{proof}

\begin{prop}
Let $G_*$ be a conic simplicial subgroup of $\mathrm{VAP}_*$ such that $G_1=\mathrm{VAP}_1=VP_2$. Then $G_*=\mathrm{VAP}_*$.
\end{prop}
\begin{proof}
The proof is given by induction on the dimension $n$ of $G_n$. From the hypothesis, $G_1=\mathrm{VAP}_1$. Suppose that $G_{n-1}=\mathrm{VAP}_{n-1}=VP_n$. From the property that $VP_{n+1}=\langle \iota_n(VP_n), s_0(VP_n),\ldots,s_{n-1}(VP_n)\rangle$, we see that $G_n=\mathrm{VAP}_n=VP_{n+1}$ and hence the result.
\end{proof}

The main point for introducing the new notion of conic simplicial group is to give a new presentation of $VP_n$ using degeneracy operations (including the extra degeneracies). From the above proposition, the new generators for $VP_n$ with $n\geq 2$ are given by
$$
s_{k_{n-2}}s_{k_{n-3}}\cdots s_{k_1}\lambda_{1,2} \textrm{ and } s_{k_{n-2}}s_{k_{n-3}}\cdots s_{k_1}\lambda_{2,1}
$$
for $0\leq s_{k_1}<s_{k_2}<\cdots<s_{k_{n-2}}\leq n-1$. Let
$$
\mu_{i,j}^{k,l}=s_{n-1}s_{n-2}\cdots \hat{s}_{l-1}\cdots \hat{s}_{k-1}\cdots s_0\lambda_{i,j}
$$
for $(i,j)=(1,2)$ or $(2,1)$ and $1<k<l\leq n$. Then
$$
VP_n=\langle \mu_{1,2}^{k,l}, \mu_{2,1}^{k,l}, 1\leq k<l\leq n \rangle.
$$
The relations with $a_{i,j}$ and $b_{i,j}$, which were defined in \cite{BW} are given by
\begin{equation}
a_{k,l-k}=\mu_{1,2}^{k,l}\textrm{ and } b_{k,l-k}=\mu_{2,1}^{k,l}
\end{equation}
By direct computations, we have the degeneracy formulae
\begin{equation}
s_t(\mu_{i,j}^{k,l})=\left\{
\begin{array}{lcl}
\mu_{i,j}^{k,l}&\textrm{ if }& t\geq l\\
\mu_{i,j}^{k,l+1}&\textrm{ if }& k\leq t<l\\
\mu_{i,j}^{k+1,l+1}&\textrm{ if }& 0\leq t<k.\\
\end{array}\right.
\end{equation}
By writing it in terms of $a_{i,j}$ and $b_{i,j}$, we have
\begin{equation}
s_ka_{i,j}=\left\{
\begin{array}{lcl}
a_{i,j}&\textrm{ if }&k\geq i+j\\
a_{i,j+1}&\textrm{ if } & i\leq k<i+j\\
a_{i+1,j+1}&\textrm{ if }&0\leq k<i,\\
\end{array}\right.
\textrm{ and }
s_kb_{i,j}=\left\{
\begin{array}{lcl}
b_{i,j}&\textrm{ if }&k\geq i+j\\
b_{i,j+1}&\textrm{ if } & i\leq k<i+j\\
b_{i+1,j+1}&\textrm{ if }&0\leq k<i.\\
\end{array}\right.
\end{equation}

For obtaining a new presentation of $VP_n$ on generators $\mu_{1,2}^{k,l}$ and $\mu_{2,1}^{k,l}$, we need to rewrite the relations
\begin{equation}\label{s-on-ijk}
s_{k_{n-3}}s_{k_{n-4}}\cdots s_{k_1}(\lambda_{ki}\lambda_{kj}\lambda_{ij})=s_{k_{n-3}}s_{k_{n-4}}\cdots s_{k_1}(\lambda_{ij}\lambda_{kj}\lambda_{ki})
\end{equation}
for distinct $1\leq i,j,k\leq 3$ and $0\leq k_1<k_2<\cdots<k_{n-3}\leq n-1$, and
\begin{equation}\label{s-on-commuting-rule}
s_{k_{n-4}}s_{k_{n-5}}\cdots s_{k_1}(\lambda_{i,j})s_{k_{n-4}}s_{k_{n-5}}\cdots s_{k_1}(\lambda_{k,l})=s_{k_{n-4}}s_{k_{n-5}}\cdots s_{k_1}(\lambda_{k,l})s_{k_{n-4}}s_{k_{n-5}}\cdots s_{k_1}(\lambda_{i,j})
\end{equation}
for distinct $1\leq i,j,k,l\leq 4$ and $0\leq k_1<k_2<\cdots<k_{n-4}\leq n-1$ in terms of $\mu_{i,j}^{k,l}$.

\section{Lifting defining relations of $VP_{n-1}$ to  $VP_n$} \label{lift}

Let $n\geq 4$. Let $\mathcal{R}^V(n)$ denote the defining relations~(\ref{rel}) and ~(\ref{relation}) of $VP_n$. By applying the degeneracy homomorphism $s_t\colon VP_n\to VP_{n+1}$ to $\mathcal{R}^V(n)$, we have the following equations
\begin{align}
& s_t(\lambda_{ij})s_t(\lambda_{kl})=s_t(\lambda_{kl})s_t(\lambda_{ij}) \label{equation3.7},\\
&
s_t(\lambda_{ki})s_t(\lambda_{kj})s_t(\lambda_{ij})=s_t(\lambda_{ij})s_t(\lambda_{kj})s_t(\lambda_{ki})
\label{equation3.8}
\end{align}
in $VP_{n+1}$ for $1\leq i,j,k,l\leq n$ with distinct letters standing for distinct indices, which is denoted as $s_t(\mathcal{R}^V(n))$.

The main aim of the present section is the proof of the following

\begin{thm} \label{lift}
Let $n\geq 4$. Consider $VP_n$ as a subgroup of $VP_{n+1}$ by adding a trivial strand in the end. Then
$$
\mathcal{R}^V(n)\cup\bigcup_{i=0}^{n-1}s_i(\mathcal{R}^V(n))
$$
gives the full set of the defining relations for $VP_{n+1}$.
\end{thm}

We will use the following proposition.

\begin{prop} \label{p3.1}
The degeneracy map $s_j : VP_n \longrightarrow VP_{n+1}$, $j = 0, 1, \ldots, n-1$, acts on the generators $\lambda_{k,l}$ and $\lambda_{l,k}$, $1 \leq k < l \leq n$, of $VP_n$ by the rules
$$
s_{i-1} (\lambda_{k,l}) = \left\{
\begin{array}{lr}
\lambda_{k+1,l+1} & for  ~i < k,\\
\lambda_{k,l+1} \lambda_{k+1,l+1} & for ~i = k, \\
\lambda_{k,l+1} & for  ~k < i < l,\\
& \\
\lambda_{k,l+1} \, \lambda_{k,l} & for ~i = l, \\
& \\
\lambda_{k,l} & for  ~i > l,
\end{array}
\right.
$$
$$
s_{i-1} (\lambda_{l,k}) = \left\{
\begin{array}{lr}
\lambda_{l+1,k+1} & for  ~i < k,\\
\lambda_{l+1,k+1} \lambda_{l+1,k} & for ~i = k, \\
\lambda_{l+1,k} & for  ~k < i < l,\\
& \\
\lambda_{l,k} \, \lambda_{l+1,k}   & for ~i = l, \\
& \\
\lambda_{l,k} & for  ~i > l.
\end{array}
\right.
$$

\end{prop}

\subsection{Lifting defining relations of $VP_{3}$ to $VP_4$}

In the group $VP_3$ we have 6 relations:
 $$
\lambda_{12} \lambda_{13} \lambda_{23} = \lambda_{23} \lambda_{13} \lambda_{12},~~~
\lambda_{21} \lambda_{23} \lambda_{13} = \lambda_{13} \lambda_{23} \lambda_{21},~~~
\lambda_{13} \lambda_{12} \lambda_{32} = \lambda_{32} \lambda_{12} \lambda_{13},
$$
$$
\lambda_{31} \lambda_{32} \lambda_{12} = \lambda_{12} \lambda_{32} \lambda_{31},~~~
\lambda_{23} \lambda_{21} \lambda_{31} = \lambda_{31} \lambda_{21} \lambda_{23},~~~
\lambda_{32} \lambda_{31} \lambda_{21} = \lambda_{21} \lambda_{31} \lambda_{32}.
$$
Acting on these relations by degeneracy map $s_2$ we get 6 relations in $VP_4$. Let us analise these relations.

1) The image of the first relation  has the form
$$
\lambda_{12} \cdot \lambda_{14} (\lambda_{13} \cdot \lambda_{24})  \lambda_{23} = \lambda_{24} (\lambda_{23} \cdot  \lambda_{14}) \lambda_{13} \cdot \lambda_{12}.
$$
Using the commutativity relation
$$
\lambda_{13}  \lambda_{24} = \lambda_{24}  \lambda_{13},~~~\lambda_{23} \lambda_{14} = \lambda_{14} \lambda_{23},
$$
we get
$$
\lambda_{12}  \lambda_{14} \lambda_{24} \cdot \lambda_{13}  \lambda_{23} = \lambda_{24} \lambda_{14}   (\lambda_{23} \lambda_{13}  \lambda_{12}).
$$
Using the following relation of $VP_3$:
$$
\lambda_{23} \lambda_{13} \lambda_{12} = \lambda_{12} \lambda_{13} \lambda_{23},
$$
we get
$$
\lambda_{12} \lambda_{14} \lambda_{24} = \lambda_{24} \lambda_{14} \lambda_{12}.
$$
that is the long relation in $VP_4$.

2) The image of the second relation has the form
$$
 \lambda_{21} \cdot \lambda_{24} (\lambda_{23} \cdot  \lambda_{14})  \lambda_{13} = \lambda_{14} (\lambda_{13} \cdot  \lambda_{24}) \lambda_{23} \cdot \lambda_{21}.
$$
Using the the commutativity relations
$$
\lambda_{23} \lambda_{14} = \lambda_{14} \lambda_{23},~~~\lambda_{13} \lambda_{24} = \lambda_{24} \lambda_{13},
$$
we get
$$
 \lambda_{21}  \lambda_{24} \lambda_{14} \lambda_{23}  \lambda_{13} = \lambda_{14} \lambda_{24} (\lambda_{13} \lambda_{23} \lambda_{21}).
$$
From the  relation of $VP_3$:
$$
 \lambda_{13} \lambda_{23} \lambda_{21} =  \lambda_{24}  \lambda_{14} \lambda_{12},
$$
we get
$$
 \lambda_{21} \lambda_{24} \lambda_{14} =  \lambda_{14}  \lambda_{24} \lambda_{21},
$$
i.e. the long relation in $VP_4$.

3) The image of the third  relation  has the form
$$
 \lambda_{14} (\lambda_{13} \cdot \lambda_{12} \cdot \lambda_{32})  \lambda_{42}  = \lambda_{32} \lambda_{42} \cdot \lambda_{12} \cdot  \lambda_{14} \lambda_{13}.
$$
Using the following relation from $VP_3$:
$$
 \lambda_{13} \lambda_{12} \lambda_{32} =  \lambda_{32}  \lambda_{12} \lambda_{13},
$$
we get
$$
(\lambda_{14} \lambda_{32})  \lambda_{12} (\lambda_{13}  \lambda_{42})  = \lambda_{32} \lambda_{42}  \lambda_{12}   \lambda_{14} \lambda_{13}.
$$
Using the commutativity relations
$$
\lambda_{14} \lambda_{32} = \lambda_{32} \lambda_{14},~~~\lambda_{13} \lambda_{42} = \lambda_{42} \lambda_{13},
$$
we have
$$
\lambda_{32} \lambda_{14}  \lambda_{12} \lambda_{42} \lambda_{13}  = \lambda_{32} \lambda_{42}  \lambda_{12}  \lambda_{14} \lambda_{13}.
$$
After cancellation
we get
$$
 \lambda_{14} \lambda_{12} \lambda_{42} =  \lambda_{42}  \lambda_{12} \lambda_{14},
$$
i.e. the long relation in $VP_4$.

4) The image of the forth  relation  has the form
$$
 \lambda_{31} (\lambda_{41} \cdot \lambda_{32})  \lambda_{42} \cdot \lambda_{12} = \lambda_{12} \cdot \lambda_{32} (\lambda_{42} \cdot \lambda_{31})  \lambda_{41}.
$$
Using the commutativity relations
$$
\lambda_{41} \lambda_{32} = \lambda_{32} \lambda_{41},~~~\lambda_{42} \cdot \lambda_{31} = \lambda_{31} \cdot \lambda_{42},
$$
we get
$$
 \lambda_{31} \lambda_{32} \lambda_{41}  \lambda_{42}  \lambda_{12} = (\lambda_{12}  \lambda_{32} \lambda_{31}) \lambda_{42}  \lambda_{41}.
$$
Using the following relation from $VP_3$:
$$
 \lambda_{12} \lambda_{32} \lambda_{31} =  \lambda_{31}  \lambda_{32} \lambda_{12},
$$
after cancelations we get
$$
 \lambda_{41} \lambda_{42} \lambda_{12} =  \lambda_{12}  \lambda_{42} \lambda_{41},
$$
i.e. the long relation in $VP_4$.

5) The image of the firth relation  has the form
$$
 \lambda_{24} (\lambda_{23} \cdot \lambda_{21} \cdot \lambda_{31}) \lambda_{41} =  \lambda_{31} \lambda_{41} \cdot \lambda_{21} \cdot   \lambda_{24} \lambda_{23}.
$$
Using the following relation from $VP_3$:
$$
 \lambda_{23} \lambda_{21} \lambda_{31} =  \lambda_{31}  \lambda_{21} \lambda_{23},
$$
and the commutativity relations
$$
\lambda_{24} \lambda_{13} = \lambda_{13} \lambda_{24},~~~\lambda_{23} \lambda_{41} = \lambda_{41} \lambda_{23},
$$
we get
$$
 \lambda_{24} \lambda_{21} \lambda_{41} =  \lambda_{41}  \lambda_{21} \lambda_{24},
$$
i.e. the long relation in $VP_4$.

6) The image of the sixth  relation  has the form
$$
 \lambda_{32} (\lambda_{42} \cdot  \lambda_{31}) \lambda_{41} \cdot \lambda_{21} = \lambda_{21} \cdot \lambda_{31} (\lambda_{41} \cdot  \lambda_{32}) \lambda_{42}.
$$
Using the commutativity relations
$$
\lambda_{42} \lambda_{31} = \lambda_{31} \lambda_{42},~~~\lambda_{41} \lambda_{32} = \lambda_{32} \lambda_{41},
$$
we get
$$
 \lambda_{32} \lambda_{31} \lambda_{42} \lambda_{41}  \lambda_{21} = (\lambda_{21}  \lambda_{31} \lambda_{32}) \lambda_{41} \lambda_{42}.
$$
Using the
following relation from $VP_3$:
$$
 \lambda_{21} \lambda_{31} \lambda_{32} =  \lambda_{32}  \lambda_{31} \lambda_{21},
$$
we get
$$
 \lambda_{42} \lambda_{41} \lambda_{21} =  \lambda_{21}  \lambda_{41} \lambda_{42},
$$
i.e. the long relation in $VP_4$. Hence, we proved

\begin{lem} \label{l1}
From relations $\mathcal{R}^V(3)$, relations $s_2(\mathcal{R}^V(3))$ and the commutativity relations in $\mathcal{R}^V(4)$ follows  the next set of  relations in $\mathcal{R}^V(4)$:
$$
\lambda_{12} \lambda_{14} \lambda_{24} = \lambda_{24} \lambda_{14} \lambda_{12},~~~
\lambda_{21} \lambda_{24} \lambda_{14} = \lambda_{14} \lambda_{24} \lambda_{21},~~~
\lambda_{14} \lambda_{12} \lambda_{42} = \lambda_{42} \lambda_{12} \lambda_{14},
$$
$$
\lambda_{41} \lambda_{42} \lambda_{12} = \lambda_{12} \lambda_{42} \lambda_{41},~~~
\lambda_{24} \lambda_{21} \lambda_{41} = \lambda_{41} \lambda_{21} \lambda_{24},~~~
\lambda_{42} \lambda_{41} \lambda_{21} = \lambda_{21} \lambda_{41} \lambda_{42},
$$
i.e. the set of relations where the indexes of the generators lie in the set $\{ 1, 2, 4 \}$.
\end{lem}

\medskip

Take the set $s_1(\mathcal{R}^V(3))$.

1) The image of the first relation  has the form
$$
 \lambda_{13} (\lambda_{12} \cdot \lambda_{14} \cdot \lambda_{24}) \lambda_{34} =  \lambda_{24} \lambda_{34}  \cdot \lambda_{14} \cdot \lambda_{13}  \lambda_{12}.
$$
Using the the following relation from Lemma \ref{l1}:
$$
 \lambda_{12} \lambda_{14} \lambda_{24} = \lambda_{24} \lambda_{14} \lambda_{12},
$$
we get
$$
(\lambda_{13} \lambda_{24}) \lambda_{14} (\lambda_{12} \lambda_{34}) =  \lambda_{24} \lambda_{34} \lambda_{14}  \lambda_{13}  \lambda_{12}.
$$
Using the commutativity relations
$$
\lambda_{13} \lambda_{24} = \lambda_{24} \lambda_{13},~~~\lambda_{12} \lambda_{34} = \lambda_{34} \lambda_{2}
$$
we have
$$
\lambda_{13} \lambda_{14} \lambda_{34} = \lambda_{34} \lambda_{14} \lambda_{13},
$$
i.e. the long relation in $VP_4$.

2) The image of the second relation has the form
$$
 \lambda_{21} (\lambda_{31} \cdot  \lambda_{24})  \lambda_{34} \cdot   \lambda_{14} = \lambda_{14}  \cdot  \lambda_{24} (\lambda_{34} \cdot  \lambda_{21})  \lambda_{31}.
$$
Using the commutativity relations
$$
\lambda_{31} \lambda_{24} = \lambda_{24} \lambda_{31},~~~\lambda_{34} \lambda_{21} = \lambda_{21} \lambda_{34}
$$
we have
$$
 \lambda_{21} \lambda_{24} \lambda_{31}  \lambda_{34}  \lambda_{14} = (\lambda_{14}  \lambda_{24} \lambda_{21}) \lambda_{34}  \lambda_{31}.
$$
Using the second relation from Lemma \ref{l1}:
$$
\lambda_{14}  \lambda_{24} \lambda_{21} =  \lambda_{21}  \lambda_{24} \lambda_{14},
$$
we get
$$
 \lambda_{31} \lambda_{34} \lambda_{14} =  \lambda_{14}  \lambda_{34} \lambda_{31},
$$
i.e. the long relation in $VP_4$.

3) The image of the third  relation has the form
$$
 \lambda_{14} \cdot  \lambda_{13} (\lambda_{12} \cdot   \lambda_{43})  \lambda_{42} = \lambda_{43}  (\lambda_{42} \cdot  \lambda_{13}) \lambda_{12} \cdot  \lambda_{14}.
$$
Using the commutativity relations
$$
\lambda_{12} \lambda_{43} = \lambda_{43} \lambda_{12},~~~\lambda_{42} \lambda_{13} = \lambda_{13} \lambda_{42}
$$
we have
$$
 \lambda_{14} \lambda_{13} \lambda_{43}  \lambda_{12}  \lambda_{42} = \lambda_{43}  \lambda_{13} (\lambda_{42} \lambda_{12}  \lambda_{14}).
$$
Using the third relation from Lemma \ref{l1}:
$$
\lambda_{42} \lambda_{12}  \lambda_{14} =  \lambda_{14} \lambda_{12}  \lambda_{42},
$$
we get
$$
 \lambda_{14} \lambda_{13} \lambda_{43} =  \lambda_{43}  \lambda_{13} \lambda_{14},
$$
i.e. the long relation in $VP_4$.

4) The image of the forth  relation has the form
$$
 \lambda_{41} \cdot \lambda_{43} (\lambda_{42} \cdot  \lambda_{13})  \lambda_{12} = \lambda_{13} (\lambda_{12} \cdot  \lambda_{43}) \lambda_{42} \cdot  \lambda_{41}.
$$
Using the commutativity relations
$$
\lambda_{42} \lambda_{13} = \lambda_{13} \lambda_{42},~~~\lambda_{12} \lambda_{43} = \lambda_{43} \lambda_{12}
$$
we have
$$
 \lambda_{41} \lambda_{43} \lambda_{13} \lambda_{42}  \lambda_{12} = \lambda_{13} \lambda_{43} (\lambda_{12} \lambda_{42}  \lambda_{41}).
$$
Using the forth relation from Lemma \ref{l1}:
$$
\lambda_{12} \lambda_{42}  \lambda_{41} = \lambda_{41} \lambda_{42}  \lambda_{12},
$$
we get
$$
 \lambda_{41} \lambda_{43} \lambda_{13} =  \lambda_{13}  \lambda_{43} \lambda_{41},
$$
i.e. the long relation in $VP_4$.

5) The image of the firth  relation has the form
$$
 \lambda_{24} (\lambda_{34} \cdot  \lambda_{21})  \lambda_{31} \cdot  \lambda_{41} = \lambda_{41} \cdot  \lambda_{21} (\lambda_{31} \cdot  \lambda_{24})  \lambda_{34}.
$$
Using the commutativity relations
$$
\lambda_{34} \lambda_{21} = \lambda_{21} \lambda_{34},~~~\lambda_{31} \lambda_{24} = \lambda_{24} \lambda_{31}
$$
we have
$$
 \lambda_{24} \lambda_{21} \lambda_{34}  \lambda_{31} \lambda_{41} = (\lambda_{41} \lambda_{21} \lambda_{24})  \lambda_{31}  \lambda_{34}.
$$
Using the firth relation from Lemma \ref{l1}:
$$
 \lambda_{41} \lambda_{21} \lambda_{24} =  \lambda_{24}  \lambda_{21} \lambda_{41},
$$
we get
$$
 \lambda_{34} \lambda_{31} \lambda_{41} =  \lambda_{41}  \lambda_{31} \lambda_{34},
$$
i.e. the long relation in $VP_4$.

6) The image of the sixth  relation has the form
$$
 \lambda_{43} (\lambda_{42} \cdot \lambda_{41} \cdot \lambda_{21}) \lambda_{31} = \lambda_{21} \lambda_{31} \cdot \lambda_{41} \cdot \lambda_{43} \lambda_{42}.
$$
using the
sixth relation from Lemma \ref{l1}:
$$
 \lambda_{42} \lambda_{41} \lambda_{21} =  \lambda_{21}  \lambda_{41} \lambda_{42},
$$
we get
$$
(\lambda_{43} \lambda_{21})  \lambda_{41} (\lambda_{42} \lambda_{31}) = \lambda_{21} \lambda_{31} \lambda_{41} \lambda_{43} \lambda_{42}.
$$
Using the commutativity relations
$$
\lambda_{43} \lambda_{21} = \lambda_{21} \lambda_{43},~~~\lambda_{42} \lambda_{31} = \lambda_{31} \lambda_{42},
$$
we get
$$
 \lambda_{43} \lambda_{41} \lambda_{31} =  \lambda_{31}  \lambda_{41} \lambda_{43},
$$
i.e. the long relation in $VP_4$. Hence, we proved

\begin{lem} \label{l2}
From relations $\mathcal{R}^V(3)$, relations $s_1(\mathcal{R}^V(3))$, $s_2(\mathcal{R}^V(3))$,  and commutativity relations in $\mathcal{R}^V(4)$ follows the next set of relations in $\mathcal{R}^V(4)$:
$$
\lambda_{13} \lambda_{14} \lambda_{34} = \lambda_{34} \lambda_{14} \lambda_{13},~~~
\lambda_{31} \lambda_{34} \lambda_{14} = \lambda_{14} \lambda_{34} \lambda_{31},~~~
\lambda_{14} \lambda_{13} \lambda_{43} = \lambda_{43} \lambda_{13} \lambda_{14},
$$
$$
\lambda_{41} \lambda_{43} \lambda_{13} = \lambda_{13} \lambda_{43} \lambda_{41},~~~
\lambda_{34} \lambda_{31} \lambda_{41} = \lambda_{41} \lambda_{31} \lambda_{34},~~~
\lambda_{43} \lambda_{41} \lambda_{31} = \lambda_{31} \lambda_{41} \lambda_{43},
$$
i.e. the set of relations where the indexes of the generators lie in the set $\{ 1, 3, 4 \}$.
\end{lem}

\medskip

Take the set of relations $s_0(\mathcal{R}^V(3))$.

1) The image of the first  relation has the form
$$
\lambda_{13}  (\lambda_{23} \cdot \lambda_{14}) \lambda_{24} \cdot  \lambda_{34} =  \lambda_{34} \cdot  \lambda_{14}  (\lambda_{24} \cdot  \lambda_{13})  \lambda_{23}.
$$
Using the commutativity relations
$$
\lambda_{23} \lambda_{14} = \lambda_{14}  \lambda_{23},~~~\lambda_{24} \lambda_{13} = \lambda_{13} \lambda_{24},
$$
 we rewrite it in the form
$$
\lambda_{13} \lambda_{14}  \lambda_{23} \lambda_{24}   \lambda_{34} =  (\lambda_{34}   \lambda_{14} \lambda_{13}) \lambda_{24}  \lambda_{23}.
$$
Using  the first relation from Lemma \ref{l2}:
$$
\lambda_{34}  \lambda_{14}  \lambda_{13} = \lambda_{13}  \lambda_{14}  \lambda_{34},
$$
we get
 the following
 long relation in $VP_4$:
$$
 \lambda_{23} \lambda_{24} \lambda_{34} =  \lambda_{34} \lambda_{24}  \lambda_{23}.
$$

2) The image of the sixth  relation has the form
$$
 \lambda_{32} (\lambda_{31} \cdot \lambda_{34} \cdot \lambda_{14})  \lambda_{24} = \lambda_{14}  \lambda_{24} \cdot \lambda_{34} \cdot \lambda_{32}  \lambda_{31}.
$$
Using the second relation from Lemma \ref{l2}:
$$
\lambda_{31}  \lambda_{34} \lambda_{14} =  \lambda_{14}  \lambda_{34} \lambda_{31},
$$
and  the commutativity relations
$$
\lambda_{32} \lambda_{14} = \lambda_{14} \lambda_{32},~~~\lambda_{31} \lambda_{24} = \lambda_{24} \lambda_{31}
$$
we have
$$
 \lambda_{32} \lambda_{34} \lambda_{24} = \lambda_{24}  \lambda_{34} \lambda_{32},
$$
i.e. the long relation in $VP_4$.

3) The image of the third  relation has the form
$$
 \lambda_{14} (\lambda_{24} \cdot \lambda_{13})  \lambda_{23} \cdot \lambda_{43} = \lambda_{43} \cdot \lambda_{13} (\lambda_{23} \cdot \lambda_{14})  \lambda_{24}.
$$
Using the commutativity relations
$$
\lambda_{24} \lambda_{13} = \lambda_{13} \lambda_{24},~~~\lambda_{23} \lambda_{14} = \lambda_{14} \lambda_{23},
$$
we have
$$
 \lambda_{14} \lambda_{13} \lambda_{24}  \lambda_{23}  \lambda_{43} = (\lambda_{43}  \lambda_{13} \lambda_{14}) \lambda_{23}  \lambda_{24}.
$$
Using the third relation from Lemma \ref{l2}:
$$
\lambda_{43}  \lambda_{13} \lambda_{14} =  \lambda_{14}  \lambda_{13} \lambda_{43},
$$
we get
$$
 \lambda_{24} \lambda_{23} \lambda_{43} =  \lambda_{43}  \lambda_{23} \lambda_{24},
$$
i.e. the long relation in $VP_4$.

4) The image of the forth  relation has the form
$$
 \lambda_{42} (\lambda_{41} \cdot \lambda_{43} \cdot \lambda_{13})  \lambda_{23} = \lambda_{13} \lambda_{23} \cdot \lambda_{43} \cdot \lambda_{42}  \lambda_{41}.
$$
Using the forth relation from Lemma \ref{l2}:
$$
\lambda_{41} \lambda_{43} \lambda_{13} = \lambda_{13} \lambda_{43}  \lambda_{41},
$$
we get
$$
 (\lambda_{42} \lambda_{13}) \lambda_{43} (\lambda_{41}  \lambda_{23}) = \lambda_{23} \lambda_{43} \lambda_{42}  \lambda_{41}.
$$
Using the commutativity relations
$$
\lambda_{42} \lambda_{13} = \lambda_{13} \lambda_{42},~~~\lambda_{41} \lambda_{23} = \lambda_{23} \lambda_{41}
$$
we get
$$
 \lambda_{42} \lambda_{43} \lambda_{23} =  \lambda_{23}  \lambda_{43} \lambda_{42},
$$
i.e. the long relation in $VP_4$.

5) The image of the firth  relation has the form
$$
 \lambda_{34} \cdot \lambda_{32} (\lambda_{31} \cdot \lambda_{42}) \lambda_{41} = \lambda_{42} (\lambda_{41} \cdot \lambda_{32}) \lambda_{31} \cdot \lambda_{34}.
$$
Using the commutativity relations
$$
\lambda_{31}  \lambda_{42} = \lambda_{42} \lambda_{31},~~~\lambda_{41} \lambda_{32} = \lambda_{32} \lambda_{41},
$$
we get
$$
 \lambda_{34} \lambda_{32} \lambda_{42}  \lambda_{31} \lambda_{41} = \lambda_{42} \lambda_{32} (\lambda_{41} \lambda_{31} \lambda_{34}).
$$
Using the firth relation from Lemma \ref{l2}:
$$
\lambda_{41} \lambda_{31} \lambda_{34} = \lambda_{34} \lambda_{31} \lambda_{41},
$$
we get
$$
 \lambda_{34} \lambda_{32} \lambda_{42} =  \lambda_{42}  \lambda_{32} \lambda_{34},
$$
i.e. the long relation in $VP_4$.

6) The image of the sixth  relation has the form
$$
 \lambda_{43} \cdot \lambda_{42} (\lambda_{41} \cdot  \lambda_{32}) \lambda_{31} = \lambda_{32} (\lambda_{31} \cdot  \lambda_{42}) \lambda_{41} \cdot  \lambda_{43}.
$$
Using the commutativity relations
$$
\lambda_{41} \lambda_{32} = \lambda_{32} \lambda_{41},~~~\lambda_{31} \lambda_{42} = \lambda_{42} \lambda_{31},
$$
we get
$$
 \lambda_{43} \lambda_{42} \lambda_{32} \lambda_{41} \lambda_{31} = \lambda_{32} \lambda_{42} (\lambda_{31} \lambda_{41}  \lambda_{43}).
$$
Using the
sixth relation from Lemma \ref{l2}:
$$
\lambda_{31} \lambda_{41} \lambda_{43} =  \lambda_{43} \lambda_{41} \lambda_{31},
$$
we get
$$
 \lambda_{43} \lambda_{42} \lambda_{32} =  \lambda_{32}  \lambda_{42} \lambda_{43},
$$
i.e. the long relation in $VP_4$. Hence, we proved

\begin{lem}
From relations $\mathcal{R}^V(3)$, relations $s_i(\mathcal{R}^V(3))$, $i = 0, 1, 2$, and commutativity relations in $\mathcal{R}^V(4)$ follows the next set relations in $\mathcal{R}^V(4)$:
$$
\lambda_{23} \lambda_{24} \lambda_{34} = \lambda_{34} \lambda_{24} \lambda_{23},~~~
\lambda_{32} \lambda_{34} \lambda_{24} = \lambda_{24} \lambda_{34} \lambda_{32},~~~
\lambda_{24} \lambda_{23} \lambda_{43} = \lambda_{43} \lambda_{23} \lambda_{24},
$$
$$
\lambda_{42} \lambda_{43} \lambda_{23} = \lambda_{23} \lambda_{43} \lambda_{42},~~~
\lambda_{34} \lambda_{32} \lambda_{42} = \lambda_{42} \lambda_{32} \lambda_{34},~~~
\lambda_{43} \lambda_{42} \lambda_{32} = \lambda_{32} \lambda_{42} \lambda_{43},
$$
i.e. the set of relations where the indexes of the generators lie in the set $\{ 2, 3, 4 \}$.
\end{lem}

\subsection{Lifting the commutativity relations from $\mathcal{R}^V(4)$ into $\mathcal{R}^V(5)$}

We have to show that $\mathcal{R}^V(5) = \langle \mathcal{R}^V(4), s_i(\mathcal{R}^V(4)), i = 0, 1, 2, 3 \rangle$. At first consider the commutativity relations
$$
[\lambda_{i4}^*, \lambda_{kl}^*], ~~1 \leq i \leq  3,~~ 1 \leq k < l \leq 3,
$$
in $\mathcal{R}^V(4)$. We divide them on the four groups:

1-st group: $[\lambda_{34}, \lambda_{12}] = [\lambda_{24}, \lambda_{13}] = [\lambda_{14}, \lambda_{23}] = 1$;

2-nd group: $[\lambda_{34}, \lambda_{21}] = [\lambda_{24}, \lambda_{31}] = [\lambda_{14}, \lambda_{32}] = 1$;

3-d group: $[\lambda_{43}, \lambda_{21}] = [\lambda_{42}, \lambda_{31}] = [\lambda_{41}, \lambda_{32}] = 1$;

4-th group: $[\lambda_{43}, \lambda_{12}] = [\lambda_{42}, \lambda_{13}] = [\lambda_{41}, \lambda_{23}] = 1$.

Take the third  relation from the 1-st group and acting on it by $s_i$, $i = 0, 1, 2, 3$, we get the following relations:
$$
[\lambda_{15} \lambda_{25},  \lambda_{34}] = [\lambda_{15}, \lambda_{24} \lambda_{34}] = [\lambda_{15},  \lambda_{24}  \lambda_{23}] = [\lambda_{15} \lambda_{14},  \lambda_{23}] = 1.
$$
Using the commutativity relation
$$
\lambda_{14} \lambda_{23} = \lambda_{23} \lambda_{14},
$$
which hold in $VP_4$, from the last relation we have
\begin{equation} \label{cc1}
[\lambda_{15},  \lambda_{23}] = 1.
\end{equation}
With considering (\ref{cc1}) we get
$$
[\lambda_{15},  \lambda_{24}] = 1.
$$
Then from the second relation follows relation $[\lambda_{15},  \lambda_{34}] = 1$ and from the first relation follows $[\lambda_{25},  \lambda_{34}] = 1$.
Hence, we have proven

\begin{lem} \label{cl1}
From the lifting $s_i$, $i = 0, 1, 2, 3$, of the relation $[\lambda_{14},  \lambda_{23}] = 1$ and the commutativity relations in $\mathcal{R}^V(4)$ follows the commutativity relations
$$
[\lambda_{15},  \lambda_{23}] = [\lambda_{15},  \lambda_{24}] =[\lambda_{15},  \lambda_{34}] = [\lambda_{25},  \lambda_{34}] = 1,
$$
from $\mathcal{R}^V(5)$.
\end{lem}

Take the second  relation in the 1-st group and acting on it by $s_i$, $i = 0, 1, 2, 3$, we get the following relations:
$$
[\lambda_{35}, \lambda_{14} \lambda_{24}] = [\lambda_{25} \lambda_{35}, \lambda_{14}] = [\lambda_{25}, \lambda_{14}  \lambda_{13}] = [\lambda_{25} \lambda_{24},  \lambda_{13}] = 1.
$$
Using the commutativity relation
$$
\lambda_{24} \lambda_{13} = \lambda_{13} \lambda_{24},
$$
which hold in $VP_4$, from the last relation we have
\begin{equation} \label{c1}
[\lambda_{25},  \lambda_{14}] = 1.
\end{equation}
Then from the third relation follows $[\lambda_{25},  \lambda_{13}] = 1$. From the second  relation follows $[\lambda_{35},  \lambda_{14}] = 1$
 and from the first relation follows $[\lambda_{35},  \lambda_{24}] = 1$.
Hence, we have proven

\begin{lem} \label{cll2}
From the lifting $s_i$, $i = 0, 1, 2, 3$, of the relation $[\lambda_{24},  \lambda_{13}] = 1$ and the commutativity relations in $\mathcal{R}^V(4)$ follows the commutativity relations
$$
[\lambda_{25},  \lambda_{14}] = [\lambda_{25},  \lambda_{13}] =[\lambda_{35},  \lambda_{14}] = [\lambda_{35},  \lambda_{24}] = 1,
$$
from $\mathcal{R}^V(5)$.
\end{lem}

Take the first relation in the 1-st group and acting on it by $s_i$, $i = 0, 1, 2, 3$, we get the following relations:
$$
[\lambda_{45}, \lambda_{13} \lambda_{23}] = [\lambda_{45}, \lambda_{13} \lambda_{12}] = [\lambda_{35} \lambda_{45},  \lambda_{12}] = [ \lambda_{35} \lambda_{34},  \lambda_{12}] = 1.
$$
Using the commutativity relation $[\lambda_{34},  \lambda_{12}] = 1$, rewrite the last relation in the form
$$
[\lambda_{35}, \lambda_{12}] = 1.
$$
 Then from the third relation follows
$[\lambda_{45},  \lambda_{12}] = 1.$
Then from the second relation follows
$[\lambda_{45}, \lambda_{13}] = 1,$
and hence from the first relation we get
$[\lambda_{45}, \lambda_{23}] = 1.$
Hence, we have proven

\begin{lem} \label{cl2}
From the lifting $s_i$, $i = 0, 1, 2, 3$, of the relation $[\lambda_{34},  \lambda_{12}] = 1$, the commutativity relations in $\mathcal{R}^V(4)$ and relations from Lemma \ref{cll2} follow the commutativity relations
$$
[\lambda_{35},  \lambda_{12}] = [\lambda_{45},  \lambda_{12}] =[\lambda_{45},  \lambda_{13}] = [\lambda_{45},  \lambda_{23}] = 1,
$$
from $VP_5$.
\end{lem}

\bigskip

Take the 2-nd group of commutativity relations:
$$
[\lambda_{34}, \lambda_{21}] = [\lambda_{24}, \lambda_{31}] = [\lambda_{14}, \lambda_{32}] = 1.
$$
Acting on the third relation by $s_i$, $i = 0, 1, 2, 3$, we get
$$
[\lambda_{15} \lambda_{25}, \lambda_{43}] = [\lambda_{15},  \lambda_{43} \lambda_{42}] = [\lambda_{15}, \lambda_{32} \lambda_{42}] = [\lambda_{15} \lambda_{14}, \lambda_{32}] = 1.
$$
Using the commutativity relations in $VP_4$, rewrite the last relation in the form
\begin{equation} \label{com22}
[\lambda_{15}, \lambda_{32}] = 1.
\end{equation}
From  (\ref{com22}), the third relation gives  $[\lambda_{15}, \lambda_{42}] = 1$. Then from the second relation follows that $[\lambda_{15}, \lambda_{43}] = 1$ and from the first relation follows that $[\lambda_{25}, \lambda_{43}] = 1$. We proved

\begin{lem}
Lifting $s_i$, $i = 0, 1, 2, 3$, the commutativity relation $[\lambda_{14}, \lambda_{32}] = 1$ and the commutativity relations of $VP_4$ give relations
$$
[\lambda_{15}, \lambda_{32}] = [\lambda_{15}, \lambda_{42}] = [\lambda_{15}, \lambda_{43}] = [\lambda_{25}, \lambda_{43}] = 1.
$$
\end{lem}

Acting on the second relation  by $s_i$, $i = 0, 1, 2, 3$, we get
$$
[\lambda_{35}, \lambda_{42} \lambda_{41}] = [\lambda_{25} \lambda_{35}, \lambda_{41}] = [\lambda_{25}, \lambda_{31} \lambda_{41}] = [\lambda_{25} \lambda_{24}, \lambda_{31}] = 1.
$$
Using the relation $[\lambda_{24}, \lambda_{31}] = 1$ in $VP_4$ from the last relation follows that
$$
[\lambda_{25}, \lambda_{31}] = 1.
$$
Hence from the third relation follows that $[\lambda_{25}, \lambda_{41}] = 1$, from the second relation follows that $[\lambda_{35}, \lambda_{41}] = $ and from the first relation follows that $[\lambda_{35}, \lambda_{42}] = 1$. We proved

\begin{lem}
Lifting $s_i$, $i = 0, 1, 2, 3$, of the commutativity relation $[\lambda_{24}, \lambda_{31}] = 1$ and the commutativity relations of $VP_4$ give relations
$$
[\lambda_{25}, \lambda_{31}] = [\lambda_{25}, \lambda_{41}] = [\lambda_{35}, \lambda_{41}] = [\lambda_{35}, \lambda_{42}] = 1.
$$
\end{lem}

\medskip

Acting on the first relation  by $s_i$, $i = 0, 1, 2, 3$, we get
$$
[\lambda_{45}, \lambda_{32} \lambda_{31}] = [\lambda_{45}, \lambda_{21} \lambda_{31}] = [\lambda_{35} \lambda_{45}, \lambda_{21}] = [\lambda_{35} \lambda_{34}, \lambda_{21}] = 1.
$$
Using the commutativity relations in $VP_4$ from the last relation follows that $[\lambda_{35}, \lambda_{21}] = 1$. Then from the third relation we have $[\lambda_{45}, \lambda_{21}] = 1$, from the second relation: $[\lambda_{45}, \lambda_{31}] = 1$ and from the first relation: $[\lambda_{45}, \lambda_{32}] = 1$.

\begin{lem}
Lifting $s_i$, $i = 0, 1, 2, 3$, of the commutativity relation $[\lambda_{34}, \lambda_{21}] =  1$ and the commutativity relations of $VP_4$ give relations
$$
[\lambda_{35}, \lambda_{21}] = [\lambda_{45}, \lambda_{21}] = [\lambda_{45}, \lambda_{31}] = [\lambda_{45}, \lambda_{32}] = 1.
$$
\end{lem}

In $VP_5$ we have 24 commutativity  relations of the form $[\lambda_{i5}, \lambda_{kl}^*] = 1$, $\lambda_{kl}^* \in \{ \lambda_{kl}, \lambda_{lk}\}$, where $1 \leq i < 5$, $1 \leq k < l \leq 4$:
$$
[\lambda_{45}, \lambda_{12}^*] = [\lambda_{45}, \lambda_{13}^*] = [\lambda_{45}, \lambda_{23}^*] =
[\lambda_{35}, \lambda_{12}^*] = [\lambda_{35}, \lambda_{14}^*] = [\lambda_{35}, \lambda_{24}^*] = 1,
$$
$$
[\lambda_{25}, \lambda_{13}^*] = [\lambda_{25}, \lambda_{14}^*] = [\lambda_{25}, \lambda_{34}^*] =
[\lambda_{15}, \lambda_{23}^*] = [\lambda_{15}, \lambda_{24}^*] = [\lambda_{15}, \lambda_{34}^*] = 1.
$$
These relations follow from the 1-st and from the 2-nd groups of commutativity relations in $\mathcal{R}^V(4)$. The other commutativity relations from $\mathcal{R}^V(5)\setminus \mathcal{R}^V(4)$ follow by the same way from the 3-d and from the 4-th groups of relations.

\subsection{Lifting the commutativity relations from $\mathcal{R}^V(n)$ to $\mathcal{R}^V(n+1)$, $n \geq 5$}

We have to show that $\mathcal{R}^V(n+1) = \langle \mathcal{R}^V(n), s_i(\mathcal{R}^V(n)), i = 0, 1, \ldots, n-1 \rangle$. At first consider the commutativity relations
$$
[\lambda_{mn}^*, \lambda_{kl}^*], ~~1 \leq m <  n,~~ 1 \leq k < l < n,
$$
in $VP_n$, which are not commutativity relations in $VP_{n-1}$. We divide them on the four groups:

1-st group: $[\lambda_{mn}, \lambda_{kl}] = 1$;

2-nd group: $[\lambda_{mn}, \lambda_{lk}] = 1$;

3-d group: $[\lambda_{nm}, \lambda_{lk}] = 1$;

4-th group: $[\lambda_{nm}, \lambda_{kl}] = 1$.

Consider the relations from the 1-st group and divide them on some subgroups.

1) Suppose that $m < k < l < n.$

Acting on the relation $[\lambda_{mn}, \lambda_{kl}] = 1$ by $s_{n-1}$ and using Proposition \ref{p3.1}, we get the relation
$$
[\lambda_{m,n+1} \lambda_{mn}, \lambda_{kl}] = 1.
$$
Since $[\lambda_{mn}, \lambda_{kl}] = 1$ and  this relation is a relation in $VP_n$, we have relation in $VP_{n+1}$:
\begin{equation} \label{r10}
[\lambda_{m,n+1}, \lambda_{kl}] = 1.
\end{equation}

Let $i$ be such that $m < k < l < i < n.$ Acting by $s_{i-1}$ on the relation $[\lambda_{mn}, \lambda_{kl}] = 1$ we get
$$
s_{i-1} ([\lambda_{mn}, \lambda_{kl}]) = [\lambda_{m,n+1}, \lambda_{kl}] = 1
$$
that is a relation in $VP_{n+1}$.

Let $i = l$, then
$$
s_{l-1} ([\lambda_{mn}, \lambda_{kl}]) = [\lambda_{m,n+1}, \lambda_{k,l+1} \lambda_{kl}] = 1.
$$
Using the commutativity relations in $VP_n$ and relation (\ref{r10}), we have
\begin{equation} \label{r11}
[\lambda_{m,n+1}, \lambda_{k,l+1}] = 1,
\end{equation}
i.e. a commutativity relation in $VP_{n+1}$.

Let $i$ satisfies the inequality $m < k < i < l < n.$ Acting by $s_{i-1}$  we get
$$
s_{i-1} ([\lambda_{mn}, \lambda_{kl}]) = [\lambda_{m,n+1}, \lambda_{k,l+1}] = 1
$$
that is a relation in $VP_{n+1}$.

Let $i=k$. Acting by $s_{k-1}$  we get
$$
s_{k-1} ([\lambda_{mn}, \lambda_{kl}]) = [\lambda_{m,n+1}, \lambda_{k,l+1} \lambda_{k+1,l+1}] = 1.
$$
Using the  relation (\ref{r11}), we have
\begin{equation} \label{r12}
[\lambda_{m,n+1}, \lambda_{k+1,l+1}] = 1,
\end{equation}
i.e. a commutativity relation in $VP_{n+1}$.

Let $i$ satisfies the inequality $m < i < k <  l < n.$ Acting by $s_{i-1}$  we get
$$
s_{i-1} ([\lambda_{mn}, \lambda_{kl}]) = [\lambda_{m,n+1}, \lambda_{k+1,l+1}] = 1
$$
that is a relation in $VP_{n+1}$.

Let $i=m$. Acting by $s_{m-1}$  we get
$$
s_{m-1} ([\lambda_{mn}, \lambda_{kl}]) = [\lambda_{m,n+1} \lambda_{m+1,n+1},  \lambda_{k+1,l+1}] = 1.
$$
Using the  relation (\ref{r12}), we have
$$
[\lambda_{m+1,n+1}, \lambda_{k+1,l+1}] = 1,
$$
i.e. a commutativity relation in $VP_{n+1}$.

Let $i$ be satisfied the inequality $i < m <  k <  l < n.$ Acting by $s_{i-1}$  we get
$$
s_{i-1} ([\lambda_{mn}, \lambda_{kl}]) = [\lambda_{m+1,n+1}, \lambda_{k+1,l+1}] = 1
$$
that is a relation in $VP_{n+1}$.

\medskip

2) Suppose that $k < m < l < n.$

Acting on the relation $[\lambda_{mn}, \lambda_{kl}] = 1$ by $s_{n-1}$, we get the relation
$$
[\lambda_{m,n+1}\lambda_{mn}, \lambda_{kl}] = 1.
$$
Since $[\lambda_{mn}, \lambda_{kl}] = 1$ that follows from the  relations in $VP_n$, we have relation:
$$
[\lambda_{m,n+1}, \lambda_{kl}] = 1.
$$

Let $i$ be such that $k < m < l < i < n.$ Acting by $s_{i-1}$ on the relation $[\lambda_{mn}, \lambda_{kl}] = 1$ we get
\begin{equation} \label{r15}
s_{i-1} ([\lambda_{mn}, \lambda_{kl}]) = [\lambda_{m,n+1}, \lambda_{kl}] = 1
\end{equation}
that is a relation in $VP_{n+1}$.

Let $i = l$, then
$$
s_{l-1} ([\lambda_{mn}, \lambda_{kl}]) = [\lambda_{m,n+1}, \lambda_{k,l+1} \lambda_{kl}] = 1
$$
Using the commutativity relations in $VP_n$ and relation (\ref{r15}), we have
$$
[\lambda_{m,n+1}, \lambda_{k,l+1}] = 1,
$$
i.e. a commutativity relation in $VP_{n+1}$.

Let $i$ be satisfied the inequality $k < m < i < l < n.$ Acting by $s_{i-1}$  we get
\begin{equation} \label{r16}
s_i ([\lambda_{mn}, \lambda_{kl}]) = [\lambda_{m,n+1}, \lambda_{k,l+1}] = 1
\end{equation}
that is a relation in $VP_{n+1}$.

Let $i=m$. Acting by $s_{m-1}$  we get
\begin{equation} \label{r17}
s_m ([\lambda_{mn}, \lambda_{kl}]) = [\lambda_{m,n+1} \lambda_{m+1,n+1},  \lambda_{k,l+1}] = 1.
\end{equation}
Using the  relation (\ref{r16}), we have
$$
[\lambda_{m+1,n+1}, \lambda_{k,l+1}] = 1,
$$
i.e. a commutativity relation in $VP_{n+1}$.

Let $i$ be satisfied the inequality $k < i < m <  l < n.$ Acting by $s_{i-1}$  we get
\begin{equation} \label{r18}
s_{i-1} ([\lambda_{mn}, \lambda_{kl}]) = [\lambda_{m+1,n+1}, \lambda_{k,l+1}] = 1
\end{equation}
that is a relation in $VP_{n+1}$.

Let $i=k$. Acting by $s_{k-1}$  we get
\begin{equation} \label{r19}
s_{k-1} ([\lambda_{mn}, \lambda_{kl}]) = [\lambda_{m+1,n+1},  \lambda_{k,l+1} \lambda_{k+1,l+1}] = 1.
\end{equation}
Using the  relation (\ref{r18}), we have
$$
[\lambda_{m+1,n+1}, \lambda_{k+1,l+1}] = 1,
$$
i.e. a commutativity relation in $VP_{n+1}$.

Let $i$ be satisfied the inequality $i < k <  m <  l < n.$ Acting by $s_{i-1}$  we get
\begin{equation} \label{r20}
s_{i-1} ([\lambda_{mn}, \lambda_{kl}]) = [\lambda_{m+1,n+1}, \lambda_{k+1,l+1}] = 1
\end{equation}
that is a relation in $VP_{n+1}$.

\medskip

3) Suppose that $k < l < m < n.$

Acting on $[\lambda_{mn}, \lambda_{kl}] = 1,$ by $s_{n-1}$, we get the relation
$$
s_{n-1}([\lambda_{mn}, \lambda_{kl}]) = [\lambda_{m,n+1}\lambda_{mn}, \lambda_{kl}] = 1.
$$
Using commutativity  relations in $VP_n$ and the commutativity relations in $VP_{n+1}$, which were proved in 2), from our relation follows
\begin{equation} \label{r21}
[\lambda_{m,n+1}, \lambda_{kl}] = 1.
\end{equation}

Let $i$ be such that $k < l < m < i < n.$ Acting by $s_{i-1}$ on the relation $[\lambda_{mn}, \lambda_{kl}] = 1$ we get
$$
s_{i-1} ([\lambda_{mn}, \lambda_{kl}]) = [\lambda_{m,n+1}, \lambda_{kl}] = 1
$$
that is a relation in $VP_{n+1}$.

Let $i = m$, then
$$
s_{m-1} ([\lambda_{mn}, \lambda_{kl}]) = [\lambda_{m,n+1} \lambda_{m+1,n+1}, \lambda_{kl}] = 1.
$$
Using  relation (\ref{r21}), we have
\begin{equation} \label{r22}
[\lambda_{m+1,n+1}, \lambda_{k,l}] = 1,
\end{equation}
i.e. a commutativity relation in $VP_{n+1}$.

Let $i$ satisfies the inequality $k < l < i < m < n.$ Acting by $s_{i-1}$  we get
$$
s_i ([\lambda_{mn}, \lambda_{kl}]) = [\lambda_{m+1,l+1}, \lambda_{kl}] = 1
$$
that is a relation in $VP_{n+1}$.

Let $i=l$. Acting by $s_{l-1}$  we get
$$
s_{l-1} ([\lambda_{mn}, \lambda_{kl}]) = [\lambda_{m+1,n+1},  \lambda_{k,l+1} \lambda_{kl}] = 1.
$$
Using (\ref{r22}) we get
\begin{equation} \label{r23}
[\lambda_{m+1,n+1}, \lambda_{k,l+1}] = 1.
\end{equation}

Let $i$ satisfies the inequality $k < i < l <  m < n.$ Acting by $s_{i-1}$  we get
\begin{equation} \label{r24}
s_{i-1} ([\lambda_{mn}, \lambda_{kl}]) = [\lambda_{m+1,n+1}, \lambda_{k,l+1}] = 1
\end{equation}
that is a relation in $VP_{n+1}$.

Let $i=k$. Acting by $s_{k-1}$  we get
$$
s_k ([\lambda_{mn}, \lambda_{kl}]) = [\lambda_{m+1,n+1},  \lambda_{k,l+1} \lambda_{k+1,l+1}] = 1.
$$
Using the  relation (\ref{r23}), we have
$$
[\lambda_{m+1,n+1}, \lambda_{k+1,l+1}] = 1.
$$

Let $i$ be satisfied the inequality $i < k < l <  m < n.$ Acting by $s_{i-1}$  we get
$$
s_i ([\lambda_{mn}, \lambda_{kl}]) = [\lambda_{m+1,n+1}, \lambda_{k+1,l+1}] = 1.
$$

\medskip

We considered only the 1-st group of relations. The proof for the other groups is similar.

\subsection{Lifting the long relations from $\mathcal{R}^V(n)$ to $\mathcal{R}^V(n+1)$, $n \geq 4$}

Denote by $R_{ijk}$ the following set of long relations
$$
\lambda_{ij} \lambda_{ik} \lambda_{jk} = \lambda_{jk} \lambda_{ik} \lambda_{ij},~~~
\lambda_{ji} \lambda_{jk} \lambda_{ik} = \lambda_{ik} \lambda_{jk} \lambda_{ji},
$$
$$
\lambda_{ik} \lambda_{ij} \lambda_{kj} = \lambda_{kj} \lambda_{ij} \lambda_{ik},~~~
\lambda_{ki} \lambda_{kj} \lambda_{ij} = \lambda_{ij} \lambda_{kj} \lambda_{ki},
$$
$$
\lambda_{jk} \lambda_{ji} \lambda_{ki} = \lambda_{ki} \lambda_{ji} \lambda_{jk},~~~
\lambda_{kj} \lambda_{ki} \lambda_{ji} = \lambda_{ji} \lambda_{ki} \lambda_{kj},
$$
i.e. relations which contains the generators with indexes from the set $\{ i, j, k \}$.

We have to prove that relations $R_{i,j,n+1}$ follows from relations of $\mathcal{R}^V(n)$, $s_l(\mathcal{R}^V(n))$, $l = 0, 1, \ldots, n-1$ and commutativity relations of $\mathcal{R}^V(n+1)$.

\begin{thm}
The long relations $R_{i,j,n+1}$ in $\mathcal{R}^V(n+1)$ follow from the relations of $\mathcal{R}^V(n)$, $s_l(\mathcal{R}^V(n))$, $l = 0, 1, \ldots, n-1$ and commutativity relations of $\mathcal{R}^V(n+1)$.
\end{thm}

To prove this theorem we start with the following

\begin{lem} \label{lem1}
Let $n \geq 4$ and for the set of integer numbers $\{ i, j, n+1 \}$, $1 \leq i < j \leq n+1$ one of the following conditions holds

1) $i \geq 3$;

2) $j - i \geq 3$;

3) $n+1 - j \geq 3$.

Then there is an integer $k$, $1 \leq k \leq n$, such that the relations $R_{i,j,n+1} \subseteq \mathcal{R}^V(n+1)$ follows from the relations $s_{k-1}(\mathcal{R}^V(n))$.
\end{lem}

\begin{proof}
1) Suppose that the condition 1) holds. Put $k=1$ and consider the relations $R_{i-1,j-1,n}$ in $\mathcal{R}^V(n)$. It is not difficult to see that  $s_0(R_{i-1,j-1,n}) = R_{i,j,n+1}$.

2) Suppose that the condition 2) holds. Put $k=i+1$ and consider the relations $R_{i,j-1,n}$ in $\mathcal{R}^V(n)$. It is not difficult to see that  $s_{i}(R_{i,j-1,n}) = R_{i,j,n+1}$.

3) Suppose that the condition 3) holds. Put $k=j+1$ and consider the relations $R_{i,j,n}$ in $\mathcal{R}^V(n)$. It is not difficult to see that  $s_{j}(R_{i,j,n}) = R_{i,j,n+1}$.
\end{proof}

Now suppose that $i = 2$ and for the set $\{ i, j, n+1 \}$ none of the conditions of the lemma is satisfied. Take the set of relations $R_{1,j-1,n}$ and find $s_0(R_{1,j-1,n})$. The first relation in $R_{1,j-1,n}$ has the form
$$
\lambda_{1,j-1} \lambda_{1n} \lambda_{j-1,n} = \lambda_{j-1,n} \lambda_{1n} \lambda_{1,j-1}.
$$
Acting by $s_0$ we get the relation
$$
(\lambda_{1,j} \lambda_{2,j}) (\lambda_{1,n+1} \lambda_{2,n+1}) \lambda_{j,n+1} = \lambda_{j,n+1} (\lambda_{1,n+1} \lambda_{2,n+1}) (\lambda_{1,j} \lambda_{2,j}).
$$
Since $\lambda_{2,j} \lambda_{1,n+1} = \lambda_{1,n+1} \lambda_{2,j}$ and $\lambda_{2,n+1} \lambda_{1,j} = \lambda_{1,j} \lambda_{2,n+1}$, rewrite the last relation in the form
\begin{equation} \label{r30}
\lambda_{1,j} \lambda_{1,n+1} \lambda_{2,j} \lambda_{2,n+1} \lambda_{j,n+1} = (\lambda_{j,n+1} \lambda_{1,n+1} \lambda_{1,j}) \lambda_{2,n+1} \lambda_{2,j}.
\end{equation}
Take the set $\{ 1, j, n+1 \}$. Since $n \geq 4$, then for this set condition 2) or condition 3) of Lemma \ref{lem1} holds. Then the set of relation $R_{1,j,n+1}$ comes from relations of $VP_n$. In particular, the relation
$$
\lambda_{j,n+1} \lambda_{1,n+1} \lambda_{1,j} =   \lambda_{1,j} \lambda_{1,n+1} \lambda_{j,n+1}
$$
holds. Using this relation, rewrite (\ref{r30}):
$$
\lambda_{1,j} \lambda_{1,n+1} \lambda_{2,j} \lambda_{2,n+1} \lambda_{j,n+1} = (\lambda_{1,j} \lambda_{1,n+1} \lambda_{j,n+1}) \lambda_{2,n+1} \lambda_{2,j}.
$$
After cancelations we have
$$
 \lambda_{2,j} \lambda_{2,n+1} \lambda_{j,n+1} =  \lambda_{j,n+1} \lambda_{2,n+1} \lambda_{2,j}.
$$
It is the first relation from $R_{2,j,n+1}$.

The second relation in $R_{1,j-1,n}$ has the form
$$
\lambda_{j-1,1} \lambda_{j-1,n} \lambda_{1,n} = \lambda_{1,n} \lambda_{j-1,n} \lambda_{j-1,1}.
$$
Acting by $s_0$ we get the relation
$$
(\lambda_{j2} \lambda_{j1}) \lambda_{j,n+1} (\lambda_{1,n+1} \lambda_{2,n+1}) = (\lambda_{1,n+1} \lambda_{2,n+1}) \lambda_{j,n+1} (\lambda_{j2} \lambda_{j1}).
$$
As we seen before the set of relation $R_{1,j,n+1}$ holds in  $VP_{n+1}$. Using the relation
$$
\lambda_{j1} \lambda_{j,n+1} \lambda_{1,n+1} =   \lambda_{1,n+1} \lambda_{j,n+1} \lambda_{j1}
$$
rewrite our relation in the form:
$$
\lambda_{j2} (\lambda_{1,n+1} \lambda_{j,n+1} \lambda_{j1}) \lambda_{2,n+1} = \lambda_{1,n+1} \lambda_{2,n+1} \lambda_{j,n+1}\lambda_{j2} \lambda_{j1}.
$$
Using the commutativity relations $\lambda_{j2} \lambda_{1,n+1} = \lambda_{1,n+1} \lambda_{j2}$ and $\lambda_{j1} \lambda_{2,n+1} = \lambda_{2,n+1} \lambda_{j1}$ we have
$$
(\lambda_{1,n+1} \lambda_{j2}) \lambda_{j,n+1} (\lambda_{2,n+1} \lambda_{j1}) = \lambda_{1,n+1} \lambda_{2,n+1} \lambda_{j,n+1}\lambda_{j2} \lambda_{j1}.
$$
After cancelations we get
$$
 \lambda_{j2} \lambda_{j,n+1} \lambda_{2,n+1} =  \lambda_{2,n+1} \lambda_{j,n+1} \lambda_{j2}.
$$
It is the second relation from $R_{1,j,n+1}$.

The third  relation in $R_{1,j-1,n}$ has the form
$$
\lambda_{1n} \lambda_{1,j-1} \lambda_{n,j-1} = \lambda_{n,j-1} \lambda_{1,j-1} \lambda_{1n}.
$$
Acting by $s_0$ we get the relation
$$
(\lambda_{1,n+1} \lambda_{2,n+1}) (\lambda_{1j} \lambda_{2j}) \lambda_{n+1,j} = \lambda_{n+1,j} (\lambda_{1j} \lambda_{2j}) (\lambda_{1,n+1} \lambda_{2,n+1}).
$$
Since $\lambda_{2,n+1} \lambda_{1j} = \lambda_{1j} \lambda_{2,n+1}$ and $\lambda_{2j} \lambda_{1,n+1} = \lambda_{1,n+1} \lambda_{2j}$, rewrite the last relation in the form
\begin{equation} \label{r31}
\lambda_{1,n+1} \lambda_{1j} \lambda_{2,n+1}  \lambda_{2j} \lambda_{n+1,j} = (\lambda_{n+1,j} \lambda_{1j} \lambda_{1,n+1}) \lambda_{2j} \lambda_{2,n+1}.
\end{equation}
As we seen, the set of relation $R_{1,j,n+1}$ comes from relations of $VP_n$. In particular, the relation
$$
\lambda_{n+1,j} \lambda_{1j} \lambda_{1,n+1} =   \lambda_{1,n+1} \lambda_{1j} \lambda_{n+1,j}
$$
holds. Using this relation, rewrite (\ref{r31}):
$$
\lambda_{1,n+1} \lambda_{1j} \lambda_{2,n+1}  \lambda_{2j} \lambda_{n+1,j} = (\lambda_{1,n+1} \lambda_{1j} \lambda_{n+1,j}) \lambda_{2j} \lambda_{2,n+1}.
$$
After cancelations we have
$$
 \lambda_{2,n+1}  \lambda_{2j} \lambda_{n+1,j} = \lambda_{n+1,j} \lambda_{2j} \lambda_{2,n+1}.
$$
It is the third relation from $R_{1,j,n+1}$.

The forth  relation in $R_{1,j-1,n}$ has the form
$$
\lambda_{n1} \lambda_{n,j-1} \lambda_{1,j-1} = \lambda_{1,j-1} \lambda_{n,j-1} \lambda_{n1}.
$$
Acting by $s_0$ we get the relation
$$
(\lambda_{n+1,2} \lambda_{n+1,1}) \lambda_{n+1,j} (\lambda_{1j} \lambda_{2j}) = (\lambda_{1j} \lambda_{2j}) \lambda_{n+1,j} (\lambda_{n+1,2} \lambda_{n+1,1}).
$$
As we seen before the set of relation $R_{1,j,n+1}$ holds in  $VP_{n+1}$. Using the relation
$$
\lambda_{n+1,1} \lambda_{n+1,j} \lambda_{1j} = \lambda_{1j} \lambda_{n+1,j} \lambda_{n+1,1}
$$
rewrite our relation in the form:
$$
\lambda_{n+1,2} (\lambda_{1j} \lambda_{n+1,j} \lambda_{n+1,1}) \lambda_{2j} = \lambda_{1j} \lambda_{2j} \lambda_{n+1,j} \lambda_{n+1,2} \lambda_{n+1,1}.
$$
Using the commutativity relations $\lambda_{n+1,2} \lambda_{1j} = \lambda_{1j} \lambda_{n+1,2}$ and $\lambda_{n+1,1} \lambda_{2j} = \lambda_{2j} \lambda_{n+1,1}$ we have
$$
(\lambda_{1j} \lambda_{n+1,2}) \lambda_{n+1,j} (\lambda_{2j} \lambda_{n+1,1}) = \lambda_{1j} \lambda_{2j} \lambda_{n+1,j} \lambda_{n+1,2} \lambda_{n+1,1}.
$$
After cancelations we get
$$
 \lambda_{n+1,2} \lambda_{n+1,j} \lambda_{2j} =  \lambda_{2j} \lambda_{n+1,j} \lambda_{n+1,2}.
$$
It is the forth relation from $R_{1,j,n+1}$.

The firth  relation in $R_{1,j-1,n}$ has the form
$$
\lambda_{j-1,n} \lambda_{j-1,1} \lambda_{n1} = \lambda_{n1} \lambda_{j-1,1} \lambda_{j-1,n}.
$$
Acting by $s_0$ we get the relation
$$
\lambda_{j,n+1} (\lambda_{j2} \lambda_{j1}) (\lambda_{n+1,2} \lambda_{n+1,1})  = (\lambda_{n+1,2} \lambda_{n+1,1}) (\lambda_{j2} \lambda_{j1}) \lambda_{j,n+1}.
$$
Since $\lambda_{j1} \lambda_{n+1,2} = \lambda_{n+1,2} \lambda_{j1}$ and $\lambda_{n+1,1} \lambda_{j2} = \lambda_{j2} \lambda_{n+1,1}$, rewrite the last relation in the form
\begin{equation} \label{r32}
\lambda_{j,n+1} \lambda_{j2} (\lambda_{n+1,2} \lambda_{j1}) \lambda_{n+1,1}  = \lambda_{n+1,2} (\lambda_{j2} \lambda_{n+1,1}) \lambda_{j1} \lambda_{j,n+1}.
\end{equation}
As we note before  the set of relation $R_{1,j,n+1}$ comes from relations of $VP_n$ and in particular, the relation
$$
\lambda_{n+1,1} \lambda_{j1} \lambda_{j,n+1} =   \lambda_{j,n+1} \lambda_{j1} \lambda_{n+1,1}
$$
holds. Using this relation, rewrite (\ref{r32}):
$$
\lambda_{j,n+1} \lambda_{j2} (\lambda_{j,n+1} \lambda_{j1} \lambda_{n+1,1}  = \lambda_{n+1,2} \lambda_{j2} \lambda_{j,n+1} \lambda_{j1} \lambda_{n+1,1}.
$$
After cancelations we have
$$
\lambda_{j,n+1} \lambda_{j2} \lambda_{j,n+1} = \lambda_{j,n+1} \lambda_{j2} \lambda_{j,n+1}.
$$
It is the firth  relation from $R_{1,j,n+1}$.

The sixth  relation in $R_{1,j-1,n}$ has the form
$$
\lambda_{n,j-1} \lambda_{n1} \lambda_{j-1,1} = \lambda_{j-1,1} \lambda_{n1} \lambda_{n,j-1}.
$$
Acting by $s_0$ we get the relation
$$
\lambda_{n+1,j} (\lambda_{n+1,2} \lambda_{n+1,1}) (\lambda_{j2} \lambda_{j1}) = (\lambda_{j2} \lambda_{j1}) (\lambda_{n+1,2} \lambda_{n+1,1}) \lambda_{n+1,j}.
$$
Using the commutativity relations $\lambda_{n+1,1} \lambda_{j2} = \lambda_{j2} \lambda_{n+1,1}$ and $\lambda_{j1} \lambda_{n+1,2} = \lambda_{n+1,2} \lambda_{j1}$ we have
$$
\lambda_{n+1,j} \lambda_{n+1,2} \lambda_{j2} \lambda_{n+1,1} \lambda_{j1} = \lambda_{j2} \lambda_{n+1,2} (\lambda_{j1} \lambda_{n+1,1} \lambda_{n+1,j}).
$$
Using the relation
$$
\lambda_{j1} \lambda_{n+1,1} \lambda_{n+1,j} = \lambda_{n+1,j} \lambda_{n+1,1} \lambda_{j1},
$$
 rewrite our relation in the form:
$$
\lambda_{n+1,j} \lambda_{n+1,2} \lambda_{j2} \lambda_{n+1,1} \lambda_{j1} = \lambda_{j2} \lambda_{n+1,2} (\lambda_{n+1,j} \lambda_{n+1,1} \lambda_{j1}).
$$
After cancelations we get
$$
\lambda_{n+1,j} \lambda_{n+1,2} \lambda_{j2} =  \lambda_{j2} \lambda_{n+1,2} \lambda_{n+1,j}.
$$
It is the sixth  relation from $R_{2,j,n+1}$.

Hence, we have proven

\begin{lem}
Let $n \geq 4$. Acting on the relations $R_{1,j-1,n}$ of $VP_n$ by $s_0$ and using the relations, which we got in Lemma \ref{lem1}, we get relations $R_{2,j,n+1}$ in $VP_{n+1}$.
\end{lem}

Next, suppose that $i = 1$ in the set $\{ i, j, n+1 \}$. Since $n \geq 4$ and we can not use Lemma \ref{lem1} for the relations $R_{i,j,n+1}$, we see that it is possible only in the case $j = 3$, $n+1 = 5$. Hence we have to prove that the relations $R_{1,3,5}$ follow from relations $s_k(\mathcal{R}^V(4))$ for some $k$.

Consider relations $R_{1,2,4}$ in $VP_4$ and acting on them by $s_1$. The first relation in $R_{1,2,4}$ has the form
$$
\lambda_{12} \lambda_{14} \lambda_{24} = \lambda_{24} \lambda_{14} \lambda_{12}.
$$
Acting on it by $s_1$ we get
$$
(\lambda_{13} \lambda_{12}) \lambda_{15} (\lambda_{25} \lambda_{35}) = (\lambda_{25} \lambda_{35}) \lambda_{15} (\lambda_{13} \lambda_{12}).
$$
Note that relations $R_{1,2,5}$ satisfy condition 3) in Lemma \ref{lem1}. Using the first relation from this set:
$$
\lambda_{12} \lambda_{15} \lambda_{25} = \lambda_{25} \lambda_{15} \lambda_{12},
$$
we get
$$
\lambda_{13} (\lambda_{25} \lambda_{15} \lambda_{12}) \lambda_{35} = \lambda_{25} \lambda_{35} \lambda_{15} \lambda_{13} \lambda_{12}.
$$
Using the commutativity relations $\lambda_{13} \lambda_{25} = \lambda_{25} \lambda_{13}$ and $\lambda_{12} \lambda_{35} = \lambda_{35} \lambda_{12}$, we have
$$
(\lambda_{25} \lambda_{13}) \lambda_{15} (\lambda_{35} \lambda_{12}) = \lambda_{25} \lambda_{35} \lambda_{15} \lambda_{13} \lambda_{12}.
$$
After cancelation we arrive to the relation
$$
\lambda_{13} \lambda_{15} \lambda_{35} = \lambda_{35} \lambda_{15} \lambda_{13}.
$$
This is the first relation from $R_{1,3,5}$.

The second relation in $R_{1,2,4}$ has the form
$$
\lambda_{21} \lambda_{24} \lambda_{14} = \lambda_{14} \lambda_{24} \lambda_{21}.
$$
Acting on it by $s_1$ we get
$$
(\lambda_{21} \lambda_{31}) (\lambda_{25} \lambda_{35}) \lambda_{15} = \lambda_{15} (\lambda_{25} \lambda_{35}) (\lambda_{21} \lambda_{31}).
$$
Using the commutativity relation $\lambda_{31} \lambda_{25} = \lambda_{25} \lambda_{31}$ and $\lambda_{35} \lambda_{21} = \lambda_{21} \lambda_{35}$, we have
$$
\lambda_{21} (\lambda_{25} \lambda_{31}) \lambda_{35} \lambda_{15} = \lambda_{15} \lambda_{25} (\lambda_{21} \lambda_{35}) \lambda_{31}.
$$
By Lemma \ref{lem1} we have relation
$$
\lambda_{15} \lambda_{25} \lambda_{21} = \lambda_{21} \lambda_{25} \lambda_{15}.
$$
Using it we get
$$
\lambda_{21} \lambda_{25} \lambda_{31} \lambda_{35} \lambda_{15} = (\lambda_{21} \lambda_{25} \lambda_{15}) \lambda_{35} \lambda_{31}.
$$
After cancelation we arrive to the relation
$$
\lambda_{31} \lambda_{35} \lambda_{15} = \lambda_{15} \lambda_{35} \lambda_{31}.
$$
This is the second relation from $R_{1,3,5}$.

Using the third relation in the set $R_{1,2,4}$:
$$
\lambda_{14} \lambda_{12} \lambda_{42} = \lambda_{42} \lambda_{12} \lambda_{14}
$$
and acting by $s_1$ we get
$$
\lambda_{15} (\lambda_{13} \lambda_{12}) (\lambda_{53} \lambda_{52}) = (\lambda_{53} \lambda_{52}) (\lambda_{13} \lambda_{12}) \lambda_{15}.
$$
Using the commutativity relation $\lambda_{12} \lambda_{53} = \lambda_{53} \lambda_{12}$ and $\lambda_{52} \lambda_{13} = \lambda_{13} \lambda_{52}$, we have
$$
\lambda_{15} \lambda_{13} (\lambda_{53} \lambda_{12}) \lambda_{52} = \lambda_{53} (\lambda_{13} \lambda_{52}) \lambda_{12} \lambda_{15}.
$$
Using the relation
$$
\lambda_{52} \lambda_{12} \lambda_{15} = \lambda_{15} \lambda_{12} \lambda_{52},
$$
which we have by Lemma \ref{lem1} we get
$$
\lambda_{15} \lambda_{13} \lambda_{53} \lambda_{12} \lambda_{52} = \lambda_{53} \lambda_{13} (\lambda_{15} \lambda_{12} \lambda_{52}).
$$
After cancelation we arrive to the relation
$$
\lambda_{15} \lambda_{13} \lambda_{53} = \lambda_{53} \lambda_{13} \lambda_{15}.
$$
This is the third relation in $R_{1,3,5}$.

The forth relation in $R_{1,2,4}$ has the form
$$
\lambda_{41} \lambda_{42} \lambda_{12} = \lambda_{12} \lambda_{42} \lambda_{41}.
$$
Acting on it by $s_1$ we get
$$
\lambda_{51} (\lambda_{53} \lambda_{52}) (\lambda_{13} \lambda_{12}) =  (\lambda_{13} \lambda_{12}) (\lambda_{53} \lambda_{52}) \lambda_{51}.
$$
Using the commutativity relation $\lambda_{52} \lambda_{13} = \lambda_{13} \lambda_{52}$ and $\lambda_{12} \lambda_{53} = \lambda_{53} \lambda_{12}$, we have
$$
\lambda_{51} \lambda_{53} (\lambda_{13} \lambda_{52}) \lambda_{12} =  \lambda_{13} (\lambda_{53} \lambda_{12}) \lambda_{52} \lambda_{51}.
$$
By Lemma \ref{lem1} we have relation
$$
\lambda_{12} \lambda_{52} \lambda_{51} = \lambda_{51} \lambda_{52} \lambda_{12}.
$$
Using it we get
$$
\lambda_{51} \lambda_{53} \lambda_{13} \lambda_{52} \lambda_{12} =  \lambda_{13} \lambda_{53} (\lambda_{51} \lambda_{52} \lambda_{12}).
$$
After cancelation we arrive to the relation
$$
\lambda_{51} \lambda_{53} \lambda_{13}  = \lambda_{13} \lambda_{53} \lambda_{51}.
$$
This is the forth relation in $R_{1,3,5}$.

Using the firth relation in the set $R_{1,2,4}$:
$$
\lambda_{24} \lambda_{21} \lambda_{41} = \lambda_{41} \lambda_{21} \lambda_{24}
$$
and acting by $s_1$ we get
$$
(\lambda_{25} \lambda_{35}) (\lambda_{21} \lambda_{31}) \lambda_{51} = \lambda_{51} (\lambda_{21} \lambda_{31}) (\lambda_{25} \lambda_{35}).
$$
Using the commutativity relation $\lambda_{35} \lambda_{21} = \lambda_{21} \lambda_{35}$ and $\lambda_{31} \lambda_{25} = \lambda_{25} \lambda_{31}$, we have
$$
\lambda_{25} (\lambda_{21} \lambda_{35}) \lambda_{31} \lambda_{51} = \lambda_{51} \lambda_{21} (\lambda_{25} \lambda_{31}) \lambda_{35}.
$$
Using the relation
$$
\lambda_{51} \lambda_{21} \lambda_{25} = \lambda_{25} \lambda_{21} \lambda_{51},
$$
which we have by Lemma \ref{lem1} we get
$$
\lambda_{25} \lambda_{21} \lambda_{35} \lambda_{31} \lambda_{51} = (\lambda_{25} \lambda_{21} \lambda_{51}) \lambda_{31} \lambda_{35}.
$$
After cancelation we arrive to the relation
$$
\lambda_{35} \lambda_{31} \lambda_{51} = \lambda_{51} \lambda_{31} \lambda_{35}.
$$
This is the firth relation from $R_{1,3,5}$.

The sixth relation in $R_{1,2,4}$ has the form
$$
\lambda_{42} \lambda_{41} \lambda_{21} = \lambda_{21} \lambda_{41} \lambda_{42}.
$$
Acting on it by $s_1$ we get
$$
(\lambda_{53} \lambda_{52}) \lambda_{51} (\lambda_{21} \lambda_{31}) = (\lambda_{21} \lambda_{31}) \lambda_{51} (\lambda_{53} \lambda_{52}).
$$
By Lemma \ref{lem1} we have relation
$$
\lambda_{52} \lambda_{51} \lambda_{21} = \lambda_{21} \lambda_{51} \lambda_{52},
$$
from which
$$
\lambda_{53} (\lambda_{21} \lambda_{51} \lambda_{52}) \lambda_{31} = \lambda_{21} \lambda_{31} \lambda_{51} \lambda_{53} \lambda_{52}.
$$
Using the commutativity relation $\lambda_{53} \lambda_{21} = \lambda_{21} \lambda_{53}$ and $\lambda_{52} \lambda_{31} = \lambda_{31} \lambda_{52}$, we have
$$
(\lambda_{21} \lambda_{53}) \lambda_{51} (\lambda_{31} \lambda_{52}) = \lambda_{21} \lambda_{31} \lambda_{51} \lambda_{53} \lambda_{52}.
$$
After cancelation we arrive to the relation
$$
\lambda_{53} \lambda_{51} \lambda_{31}  = \lambda_{31} \lambda_{51} \lambda_{53}.
$$
This is the sixth relation from $R_{1,3,5}$.

\subsection{Simplicial group $T_*$} \label{T}

The simplicial group $T_*$ was defined in the paper \cite{BW}. In the same paper  was  proved that  $T_3$ is generated by elements
$$
a_{31},~~a_{22},~~a_{13},~~b_{31},~~b_{22},~~b_{13}
$$
and is defined by relations
$$
[a_{31}, a_{22}]^{c_{11}^k c_{21}^m} = [a_{31}, a_{13}]^{c_{11}^k c_{21}^m} = [a_{22}, a_{13}]^{c_{11}^k c_{21}^m} = 1,
$$
$$
[b_{31}, b_{22}]^{c_{11}^k c_{21}^m} = [b_{31}, b_{13}]^{c_{11}^k c_{21}^m} = [b_{22}, b_{13}]^{c_{11}^k c_{21}^m} = 1,
$$
 that can be written in the form
$$
[a_{31}, a_{22}^{c_{22}^{m} c_{31}^{-m}}] = [a_{31}, a_{13}^{c_{13}^{k} c_{22}^{m-k} c_{31}^{-m}}] = [a_{22}^{c_{22}^m c_{31}^{-m}}, a_{13}^{c_{13}^k c_{22}^{m-k} c_{31}^{-m}}] = 1,
$$
$$
[b_{31}, b_{22}^{c_{22}^{m} c_{31}^{-m}}] = [b_{31}, b_{13}^{c_{13}^{k} c_{22}^{m-k} c_{31}^{-m}}] = [b_{22}^{c_{22}^m c_{31}^{-m}}, b_{13}^{c_{13}^k c_{22}^{m-k} c_{31}^{-m}}] = 1.
$$
where $k, m \in \mathbb{Z}$.

In the general case we will prove

\begin{thm}\label{T_n}
The group $T_n$, $n \geq 2$ is generated by elements
$$
a_{i,n+1-i},~~b_{i,n+1-i},~~i = 1, 2, \ldots n,
$$
and is defined by relations
$$
[a_{i,n+1-i}, a_{j,n+1-j}]^{c_{11}^{k_1} c_{21}^{k_2} \ldots c_{n-1,1}^{k_{n-1}} },
$$
$$
[b_{i,n+1-i}, b_{j,n+1-j}]^{c_{11}^{k_1} c_{21}^{k_2} \ldots c_{n-1,1}^{k_{n-1}} },
$$
where $1 \leq i \not= j \leq n$, $k_l \in \mathbb{Z}$.
\end{thm}

\section{$VP_n$ as a subgroup of $VB_{n+1}$} \label{conj}

In the previous section we shown how it is possible to construct $VP_n$ from $VP_{n-1}$ using operation cabling. In this section we will show how it is possible to construct $VP_{n+1}$, using the action of the symmetric group $S_{n+1} = \langle \rho_1, \rho_2, \ldots, \rho_{n_1} \rangle$, which is a subgroup of the virtual braid group $VB_{n+1} = VP_{n+1} \leftthreetimes S_{n+1}$. Recall that $S_{n+1}$ acts on the generators of $VP_{n+1}$ by the rule
$$
\rho_k \lambda_{ij} \rho_k = \lambda_{\rho_k(i),\rho_k(j)},~~~k = 1, 2, \ldots, n-1.
$$

The symmetric group $S_{n+1}$ ia s disjoint union of cosets by $S_n$:
$$
S_{n+1} = S_{n} e \sqcup S_{n} \rho_n \sqcup S_{n} \rho_n \rho_{n-1} \sqcup \ldots \sqcup S_{n} \rho_n \rho_{n-1} \ldots \rho_{1}.
$$

We will denote $\mathcal{X}_k$ the set of generators of $VP_k$, $k \geq 2$, i.e.
$$
\mathcal{X}_k = \{ \lambda_{ij} ~|~1 \leq i \not= j \leq k \};
$$
$\mathcal{R}_k$ will denote the set of defining relations of $VP_k$. In particular, $\mathcal{LR}_k$ will denote the set of long relations and $\mathcal{CR}_k$ the set of commutativity relations. It is evident that
$$
\mathcal{R}_k = \mathcal{LR}_k \cup \mathcal{CR}_k.
$$
Since, $VP_3$ does not contain commutativity relations, then $\mathcal{R}_3 = \mathcal{LR}_3.$

Let $k > 2$ and $1 \leq i < j < l \leq k$ be three  distinct integer numbers. Denote by $\mathcal{R}^{ijl}_k$ the following set of long defining relations from $\mathcal{R}_k$:
$$
\lambda_{ij} \lambda_{il} \lambda_{jl} = \lambda_{jl}  \lambda_{il} \lambda_{ij},~~~\lambda_{ji} \lambda_{jl} \lambda_{il} = \lambda_{il}  \lambda_{jl} \lambda_{ji},
$$
$$
\lambda_{il} \lambda_{ij} \lambda_{lj} = \lambda_{lj}  \lambda_{ij} \lambda_{il},~~~\lambda_{li} \lambda_{lj} \lambda_{ij} = \lambda_{ij}  \lambda_{lj} \lambda_{li},
$$
$$
\lambda_{jl} \lambda_{ji} \lambda_{li} = \lambda_{li}  \lambda_{ji} \lambda_{jl},~~~\lambda_{lj} \lambda_{li} \lambda_{ji} = \lambda_{ji}  \lambda_{li} \lambda_{lj}.
$$
Then
$$
\mathcal{LR}_k = \bigsqcup_{1 \leq i<j<l \leq k} \mathcal{R}^{ijl}_k.
$$
In particular,
$$
\mathcal{R}_3 =  \mathcal{R}^{123}_3.
$$

Let the integers $i,j,l,m \in \{ 1, 2, \ldots, k\}$ satisfy the conditions
$$
 i < j, ~~l < m,~~j > m.
$$
Denote
$$
\mathcal{R}^{i,j,l,m}_k = \{ \lambda_{ij}^* \lambda_{lm}^* = \lambda_{lm}^* \lambda_{ij}^* \}
$$
the set of four commutativity relation with fixed indexes, then
$$
\mathcal{CR}_k = \bigsqcup_{ i < j,~ l < m,~ j > m} \mathcal{R}^{i,j,l,m}_k
$$
is the full set of the  commutativity relations in $VP_k$

Take the set of generators of $VP_3$:
$$
\mathcal{X}_3 = \{ \lambda_{12},  \lambda_{21}, \lambda_{13}, \lambda_{23}, \lambda_{31}, \lambda_{32} \}
$$
and acting on it by coset representatives of $S_4$ by $S_3$ we get
$$
\mathcal{X}_3^{\rho_3} = \{ \lambda_{12},  \lambda_{21}, \lambda_{14}, \lambda_{24}, \lambda_{41}, \lambda_{42} \},
$$
$$
\mathcal{X}_3^{\rho_3 \rho_2} = \{ \lambda_{13},  \lambda_{31}, \lambda_{14}, \lambda_{34}, \lambda_{41}, \lambda_{43} \},
$$
$$
\mathcal{X}_3^{\rho_3 \rho_2 \rho_1} = \{ \lambda_{23},  \lambda_{32}, \lambda_{24}, \lambda_{34}, \lambda_{42}, \lambda_{43} \}.
$$
We see that
$$
\mathcal{X}_4 = \mathcal{X}_3 \cup \mathcal{X}_3^{\rho_3} \cup \mathcal{X}_3^{\rho_3 \rho_2}.
$$
In the general case we have the similar result

\begin{prop}
For $n \geq 3$ the following equality holds
$$
\mathcal{X}_{n+1} = \mathcal{X}_n \cup \mathcal{X}_n^{\rho_n} \cup \mathcal{X}_n^{\rho_n \rho_{n-1}}.
$$
\end{prop}

\begin{proof}
Any generator in $\mathcal{X}_{n+1} \setminus \mathcal{X}_{n}$ has the form $\lambda_{i,n+1}^*$ for some $i$, $1\leq i \leq n$. Take the generator $\lambda_{1n}^* \in  \mathcal{X}_{n}$ and acting on it by conjugation of $\rho_n$:
$$
\left( \lambda_{1n}^* \right)^{\rho_n} = \lambda_{1,n+1}^*,~~~\left( \lambda_{2n}^* \right)^{\rho_n} = \lambda_{2,n+1}^*, \ldots, \left( \lambda_{n-1,n}^* \right)^{\rho_n} = \lambda_{n-1,n+1}^*.
$$
To find the last generator $\lambda_{n,n+1}^*$, take the generator $\lambda_{n-1,n}^*$ and acting of conjugation by $\rho_n \rho_{n-1}$ we get
$$
\left( \lambda_{n-1,n}^* \right)^{\rho_n \rho_{n-1}} = \left( \lambda_{n-1,n+1}^* \right)^{\rho_{n-1}} = \lambda_{n,n+1}^*.
$$

\end{proof}

To find the set of defining relations in $\mathcal{R}_4$, take the defining relations of $\mathcal{R}_3 = \mathcal{R}^{123}$ and acting by coset representatives we get
$$
\mathcal{R}_3^{\rho_3} = \mathcal{R}_4^{124},~~~\mathcal{R}_3^{\rho_3 \rho_2} = \mathcal{R}_4^{134},~~~\mathcal{R}_3^{\rho_3 \rho_2 \rho_1} = \mathcal{R}_4^{234}.
$$
Since
$$
\mathcal{LR}_4 = \mathcal{R}_4^{123} \sqcup \mathcal{R}_4^{124} \sqcup \mathcal{R}_4^{134} \sqcup \mathcal{R}_4^{234} ~~\mbox{and}~~\mathcal{R}_4^{123} = \mathcal{R}_3^{123} = \mathcal{R}_3,
$$
we get
$$
\mathcal{LR}_4 = \mathcal{R}_3 \sqcup \mathcal{R}_3^{\rho_3} \sqcup \mathcal{R}_3^{\rho_3 \rho_2} \sqcup \mathcal{R}_3^{\rho_3 \rho_2 \rho_1}.
$$
In $VP_3$ we don't have commutativity relations hence, we have

\begin{prop}
$$
\mathcal{R}_4 = \mathcal{R}_3 \sqcup \mathcal{R}_3^{\rho_3} \sqcup \mathcal{R}_3^{\rho_3 \rho_2} \sqcup \mathcal{R}_3^{\rho_3 \rho_2 \rho_1} \sqcup \mathcal{CR}_4.
$$
\end{prop}

In the general case we can prove

\begin{thm}
For $n \geq 4$ we have
$$
\mathcal{R}_{n+1} = \mathcal{R}_n \sqcup \mathcal{R}_n^{\rho_n} \sqcup \mathcal{R}_n^{\rho_n \rho_{n-1}} \sqcup \ldots \sqcup \mathcal{R}_n^{\rho_n \rho_{n-1} \ldots \rho_1}.
$$
\end{thm}

\begin{proof}
Consider the set of long relations $\mathcal{R}_{n+1}^{i,j,n+1}$ which does not lie in $\mathcal{R}_{n}$. If $j\not= n$, then the relations $\mathcal{R}_{n}^{i,j,n}$ lie in $\mathcal{R}_{n}$ and acting by $\rho_n$ we get
$$
\left( \mathcal{R}_{n}^{i,j,n} \right)^{\rho_n} = \mathcal{R}_{n+1}^{i,j,n+1}.
$$
If $j=n$, but $i\not= n-1$, then
$$
\left( \mathcal{R}_{n}^{i,n-1,n}\right)^{\rho_n \rho_{n-1}} = \left(  \mathcal{R}_{n+1}^{i,n-1,n+1} \right)^{\rho_{n-1}} = \mathcal{R}_{n+1}^{i,n,n+1}.
$$
If $j=n$, $i = n-1$, then
$$
\left( \mathcal{R}_{n}^{n-2,n-1,n} \right)^{\rho_n \rho_{n-1} \rho_{n-2}} = \left( \mathcal{R}_{n+1}^{n-2,n-1,n+1} \right)^{\rho_{n-1} \rho_{n-2}} = \left( \mathcal{R}_{n+1}^{n-2,n,n+1} \right)^{\rho_{n-2}} = \mathcal{R}_{n+1}^{n-1,n,n+1}.
$$

Consider a set of commutativity relations
$$
\mathcal{R}^{i,n+1,l,m}_{n+1} \in \mathcal{R}_{n+1} \setminus \mathcal{R}_{n}.
$$
We will assume that $i < l < m$. Proofs for other cases is similar.

If $m \not= n$, then
$$
\left( \mathcal{R}^{i,n,l,m}_{n} \right)^{\rho_n} = \mathcal{R}^{i,n+1,l,m}_{n+1}.
$$
If $m= n$, but $l \not= n-1$, then
$$
\left( \mathcal{R}^{i,n,l,n-1}_{n} \right)^{\rho_n \rho_{n-1}} = \left( \mathcal{R}^{i,n+1,l,n-1}_{n+1} \right)^{\rho_{n-1}} = \mathcal{R}^{i,n+1,l,n}_{n+1}.
$$
If $m= n$, $l = n-1$, but $i \not= n-2$, then
$$
\left( \mathcal{R}^{i,n,n-2,n-1}_{n} \right)^{\rho_n \rho_{n-1}  \rho_{n-2}} = \left( \mathcal{R}^{i,n+1,n-2,n-1}_{n+1} \right)^{\rho_{n-1}  \rho_{n-2}} = \left( \mathcal{R}^{i,n+1,n-2,n}_{n} \right)^{ \rho_{n-2}} = \mathcal{R}^{i,n+1,n-1,n}_{n+1}.
$$
If $m= n$, $l = n-1$ and $i = n-2$, then
$$
\left( \mathcal{R}^{n-3,n,n-2,n-1}_{n} \right)^{\rho_n \rho_{n-1}  \rho_{n-2} \rho_{n-3}} = \left( \mathcal{R}^{n-3,n+1,n-2,n-1}_{n+1} \right)^{\rho_{n-1}  \rho_{n-2} \rho_{n-3}} =
$$
$$
= \left( \mathcal{R}^{n-3,n+1,n-2,n}_{n} \right)^{ \rho_{n-2} \rho_{n-3}} = \left(  \mathcal{R}^{n-3,n+1,n-1,n}_{n+1}\right)^{\rho_{n-3}} =  \mathcal{R}^{n-2,n+1,n-1,n}_{n+1}.
$$
\end{proof}

\medskip

\section{Cabling of the Artin pure braid group} \label{s41}

In the paper \cite{CW} was defined a cabling on the the set of pure braid groups $\{P_n\}_{n=2,3,\ldots}$. It was proven that in fact that all generators of $P_n$ come from the unique   generator $A_{12}$ of $U_2$, using cabling.
In this section we find a set of defining relation of $P_4$ in these generators.

In the previous section we define elements $c_{ij} = b_{ij} a_{ij}$. Put
$$
T_k^c = \langle c_{ij}~|~i+j = k+1  \rangle,~~k = 1, 2, \ldots, n-1.
$$
Any group $T_k^c$ for $k>1$ is getting from $T_{k-1}^c$ using cabling, i.e.
$$
T_k^c = \langle s_0(T_{k-1}^c), s_1(T_{k-1}^c), \ldots, s_{k-2}(T_{k-1}^c) \rangle.
$$
Then  $P_n = \langle T_1^c, T_2^c, \ldots, T_{n-1}^c \rangle$.

In the paper \cite{BW} was found  set of defining relations of $P_4$ in the cabled generators $c_{ij}$, more precisely  was proven

\begin{prop} The group $P_4$ is generated by elements
$$
c_{11},~~ c_{21},~~ c_{12},~~ c_{31},~~c_{22},~~ c_{13}
$$
and is defined by relations (where $\varepsilon = \pm 1$):
$$
c_{21}^{c_{11}^{\varepsilon}} = c_{21},~~~c_{12}^{c_{11}^{\varepsilon}} = c_{12}^{c_{21}^{-\varepsilon}},~~~c_{31}^{c_{11}^{\varepsilon}} = c_{31},~~~c_{22}^{c_{11}^{\varepsilon}} = c_{22},~~~c_{13}^{c_{11}^{\varepsilon}} = c_{13}^{c_{22}^{-\varepsilon}},
$$
$$
c_{31}^{c_{21}^{\varepsilon}} = c_{31},~~~c_{22}^{c_{21}^{\varepsilon}} = c_{22}^{c_{31}^{-\varepsilon}},~~~c_{13}^{c_{21}^{\varepsilon}} = c_{13}^{c_{22}^{\varepsilon} c_{31}^{-\varepsilon}},
$$
$$
c_{31}^{c_{12}^{\varepsilon}} = c_{31},~~~c_{13}^{c_{12}^{\varepsilon}} = c_{13}^{c_{31}^{-\varepsilon}}.
$$
$$
c_{22}^{c_{12}^{-1}} = [c_{31}, c_{13}^{-1}] \, [c_{13}^{-1}, c_{22}] \, c_{22} \, [c_{21}^2, c_{12}^{-1}] = c_{13}^{c_{31}} c_{13}^{-c_{22}}  c_{22} [c_{21}^2, c_{12}^{-1}],
$$
$$
c_{22}^{c_{12}} = [c_{12}, c_{21}^{-2}] \, c_{22} \, [c_{22}^{-3}, c_{13}]  \, [c_{13}, c_{31}^{-1}] = [c_{12}, c_{21}^{-2}] \, c_{13}^{-c_{22}^{-2}} \, c_{22} \, c_{13}^{c_{31}^{-1}} .
$$

\end{prop}

\medskip

 Define the following subgroups of $P_4$:
$$
V_1 = \langle c_{11}, c_{12}, c_{13} \rangle,~~~V_2 = \langle c_{21}, c_{22} \rangle,~~~V_3 = \langle c_{31} \rangle.
$$
Then

\begin{thm}
$$
P_4 = V_1 \leftthreetimes (V_2 \leftthreetimes V_3).
$$
\end{thm}

\begin{proof}
At first prove that $\langle V_2, V_3 \rangle = V_2 \leftthreetimes V_3$. Indeed, this group is defined by relations.
$$
[c_{31}, c_{21}] = 1,~~c_{22}^{c_{21}} = c_{22}^{c_{31}^{-1}}.
$$
Since the first relation we can write in the form
$$
c_{21}^{c_{31}} = c_{21},
$$
we have the need decomposition.

From the defining relations of $P_4$ find the following formulas of conjugation by $c_{31}$:
$$
c_{11}^{c_{31}} = c_{11},~~~c_{12}^{c_{31}} = c_{12 },~~~c_{13}^{c_{31}} = c_{13}^{c_{12}^{-1}}.
$$
Hence
$$
P_4 = \langle V_1, V_2 \rangle \leftthreetimes V_3.
$$

Find the formulas of conjugations by $c_{21}$:
$$
c_{11}^{c_{21}} = c_{11},~~~c_{12}^{c_{21}} = c_{12}^{c_{11}^{-1}},~~~c_{13}^{c_{21}} = c_{13}^{c_{12} c_{11}^{-1}}.
$$
Also we have two formulas of conjugation by $c_{22}$:
$$
c_{11}^{c_{22}} = c_{11},~~~c_{13}^{c_{22}} = c_{13}^{c_{11}^{-1}}.
$$
To  finish the proof we need to find a formula for the conjugation $c_{12}^{c_{22}}$ and $c_{12}^{c_{22}^{-1}}$.

In the proof of the previous theorem we have found relation:
$$
c_{21} c_{22}^{-1} c_{13} c_{12}^{-1} = c_{21}^{-1} c_{12}^{-1}  c_{21}^{2} c_{22}^{-1} (c_{22}^{-1} c_{13} c_{22}).
$$
Multiply both sides on $c_{21}^{-1}$ to the left and using relation
$$
c_{22}^{-1} c_{13} c_{22} = c_{11} c_{13} c_{11}^{-1},
$$
we get
$$
c_{22}^{-1} c_{13} c_{12}^{-1} = (c_{21}^{-2} c_{12}^{-1}  c_{21}^{2}) (c_{11} c_{13} c_{11}^{-1})^{c_{22}} c_{22}^{-1}.
$$
Using the conjugation formulas:
$$
c_{21}^{-2} c_{12}^{-1}  c_{21}^{2} = c_{11}^{2} c_{12}^{-1}  c_{11}^{-2},~~~(c_{11} c_{13} c_{11}^{-1})^{c_{22}} = c_{11}^2 c_{13} c_{11}^{-2},
$$
we get
$$
(c_{13} c_{12}^{-1})^{c_{22}} = c_{11}^{2} c_{12}^{-1}  c_{13} c_{11}^{-2}.
$$
Using the conjugation formula:
$$
c_{13}^{c_{22}} = c_{13}^{c_{11}^{-1}}
$$
we have
$$
c_{13} c_{11}^{-1} c_{12}^{-c_{22}} =  c_{11} c_{12}^{-1} c_{13} c_{11}^{-2}.
$$
From this relation  we get the need formula:
$$
c_{12}^{c_{22}} = c_{11}^2 c_{13}^{-1} c_{12} c_{11}^{-1} c_{13} c_{11}^{-1}.
$$

Conjugating both sides by $c_{22}^{-1}$ we find
$$
c_{12}^{c_{22}^{-1}} = c_{11}^{-1} c_{13} c_{11}^{-1} c_{12}  c_{13}^{-1} c_{11}^{2}.
$$
\end{proof}

In this theorem we used full set of defining relations for $P_4$. Let us consider the group $P_3$. It has the following presentation
$$
P_3 = \langle c_{11}, c_{21}, c_{12}~|~c_{11}^{c_{21}} = c_{11},~~c_{12}^{c_{21}} = c_{12}^{c_{11}^{-1}} \rangle.
$$
Using degeneracy maps $s_0, s_1, s_2$, we construct the following subgroups of $P_4$:
$$
s_0(P_3) = \langle c_{21}, c_{31}, c_{22}~|~c_{21}^{c_{31}} = c_{21},~~c_{22}^{c_{31}} = c_{22}^{c_{21}^{-1}} \rangle,
$$
$$
s_1(P_3) = \langle c_{12}, c_{31}, c_{13}~|~c_{12}^{c_{31}} = c_{12},~~c_{13}^{c_{31}} = c_{13}^{c_{12}^{-1}} \rangle,
$$
$$
s_2(P_3) = \langle c_{11}, c_{22}, c_{13}~|~c_{11}^{c_{22}} = c_{11},~~c_{13}^{c_{22}} = c_{13}^{c_{11}^{-1}} \rangle.
$$

%\end{document}

From the list of relations in $P_3$, $s_i(P_3)$, $i = 0, 1, 2$, we see that it is not the full list of relations for $P_4$. To have a full list we can add the relations
$$
c_{11}^{c_{31}} = c_{11},~~c_{13}^{c_{21}} = c_{13}^{c_{12} c_{11}^{-1}},~~c_{12}^{c_{22}} = c_{11}^2 \, c_{13}^{-1} \, c_{12} \, c_{11}^{-1} \, c_{13} \, c_{11}^{-1}.
$$

But us follows from Theorem \ref{lift}, for $n \geq 5$ the full list of relations for $P_n$ comes from relations of $P_{n-1}$, $s_i(P_{n-1})$, $i = 0, 1, \ldots, n-2$. Using induction by $n$ we can find relations of $P_n$. We get the following relations:

-- conjugations by $c_{n-1,1}$

$$
c_{n-k,k}^{c_{n-1,1}} = c_{n-k,k}^{c_{n-k,k-1}^{-1}},~~~k = 2, 3, \ldots, n-1;~~~c_{ij}^{c_{n-1,1}} = c_{ij}~~~\mbox{if}~i+j < n;
$$

-- conjugations by $c_{n-2,2}$

$$
c_{n-k,k}^{c_{n-2,2}} = c_{n-k,k}^{c_{n-k,k-2}^{-1}},~~~k = 2, 3, \ldots, n-1;~~~c_{ij}^{c_{n-2,2}} = c_{11}^2 \, c_{13}^{-1} \, c_{ij} \, c_{11}^{-1} \, c_{13} \, c_{11}^{-1},~~~i+j < n;
$$
$$
c_{lm}^{c_{n-2,}} = c_{lm}~~~\mbox{in all other cases};
$$

In the general case we prove

\begin{thm}
For $n \geq 3$ the pure braid group $P_n$ is the semi-direct product of free groups:
$$
P_n = V_1 \leftthreetimes (V_{2} \leftthreetimes ( \ldots ( V_{n-2} \leftthreetimes V_{n-1})\ldots)),
$$
where
$$
V_{n-1} = \langle c_{n-1,1} \rangle,
$$
$$
V_{n-2} = \langle c_{c_{n-2,1}, n-2,2}  \rangle,
$$
$$
......................................
$$
$$
V_{1} = \langle c_{11}, c_{12}, \ldots, c_{1,n-1} \rangle.
$$
\end{thm}

\begin{proof}
The theorem is true for $n=4$.
We prove that $P_n = V_1 \leftthreetimes P_{n-1}$ for $n > 4$. By the lifting theorem the set of defining relations for $P_n$ come from the set of defining relations for $P_{n-1}$ by degeneracy maps. Using this fact let us prove that $V_1$ is normal in $P_n$.
\end{proof}

\section{Directions for further research} \label{fin}

We know some generalizations oh the Artin braid group $B_n$, for example, welded braid group, singular braid groups and others (see \cite{B}). In these groups it is possible to define pure subgroups. It is interesting to study presentations of these subgroups in cabled generators, define analogs of simplicial group $T_*$ and find its homotopy type.

For example, the welded braid group $WB_n$ contains the group of basis conjugating automorphisms $Cb_n$.

\begin{question}
The group of basic conjugating automorphisms $Cb_2$ is generated by two automorphisms $\varepsilon_{21}$ and $\varepsilon_{12}$ which generate a free group of rank 2. Using operation cabling find a presentation of $Cb_n$ in the cable generators.
\end{question}

\begin{question}
Let  $\varphi : VP_n \to Cb_n$ be a homomorphism which sends $\lambda_{ij}$ to $\varepsilon_{ij}$. Is it true that $T_{n-1}$ is isomorphic to its image $\varphi(T_{n-1})$?
\end{question}

We know Artin and Gassner representations of $P_n$ (see \cite[Chapter 3]{Bir}).

\begin{question}
Find analogs of Artin and Gassner representations of $P_n$, using decomposition from Section \ref{s41}. Are they equivalent to the classical representations?
\end{question}

\end{document}